\documentclass[a4paper,11pt]{amsart}
\usepackage{amssymb,xspace,amscd}
\usepackage[pdftex]{hyperref}

\title[Complex vector fields and hypoelliptic
PDO's]{Complex vector fields and hypoelliptic 
partial differential operators}

\author[A.~Altomani]{Andrea Altomani}
\address{A.\ Altomani:
Research Unity in Mathematics\\ University of Luxembourg\\ 162a,
avenue de la Fa\"iencerie\\ L-1511 Luxembourg}
\email{andrea.altomani@uni.lu}

\author[C.~D.~Hill]{C.~Denson Hill}
\address{C.D.Hill: Department of Mathematics\\
Stony Brook University\\ Stony Brook, NY 11794 \\USA}
\email{dhill@math.sunysb.edu}

\author[M.~Nacinovich]{Mauro Nacinovich}
\address{M.\ Nacinovich:
Dipartimento di Matematica\\ II Universit\`a di Roma
``Tor Ver\-ga\-ta''\\ Via della Ricerca Scientifica\\ 00133 Roma
\\Italy}
\email{nacinovi@mat.uniroma2.it}

\author[E.~Porten]{Egmont Porten}
\address{E.\ Porten: Department of Mathematics\\ 
Mid Sweden University\\ 85170 Sundsvall \\Sweden}
\email{Egmont.Porten@miun.se}

\date{\today}

\numberwithin{equation}{section}

\theoremstyle{plain}
\newtheorem{thm}{Theorem}[section]
\newtheorem{lem}[thm]{Lemma}
\newtheorem{cor}[thm]{Corollary}
\newtheorem{prop}[thm]{Proposition}

\theoremstyle{definition}

\newtheorem{defn}[thm]{Definition}
\newtheorem{rmk}[thm]{Remark}

\begin{document}
\begin{abstract}
We prove a subelliptic estimate for systems of complex vector fields
under some assumptions that generalize the essential pseudoconcavity 
for
$CR$ manifolds, that was first introduced by two of the Authors,
and the H\"ormander's bracket condition 
for real vector fields.
Applications are given to prove the hypoellipticity of first order systems
and second order partial differential operators. Finally we describe
a class of compact homogeneous $CR$ manifolds for which the distribution
of $(0,1)$ vector fields satisfies a subelliptic estimate.
\end{abstract}
\maketitle

\tableofcontents
\section*{Introduction}
In this paper we prove a subelliptic estimate
for sums of squares of complex vector fields. 
Namely, given a distribution $\mathfrak{Z}(M)$ of complex vector fields
on an $m$ dimensional real smooth manifold $M$, 
we find conditions for 
the subellipticity of $\mathfrak{Z}(M)$, i.e. 
the validity of the estimate
\begin{equation}
  \label{eq:ob}
  \|u\|_{\epsilon}^2\leq{C}\left(\|u\|^2_0
+\sum_{j=1}^n\|L_j(u)\|_0^2
\right)\quad\forall{u}\in\mathcal{C}^{\infty}_0(U),
\end{equation}
where $U$ is a relatively compact open subset 
of $M$ and $\epsilon>0$,
$C>0$ are positive constants,
depending on $U$ and on
$L_1,\hdots,L_n\in\mathfrak{Z}(M)$, and $\mathcal{C}^{\infty}_0(U)$
is the space of smooth complex valued functions on $M$, with compact
support contained in~$U$. It is known that this estimate 
implies the $\mathcal{C}^{\infty}$-hypoellipticity in $U$ of
$\sum_{j=1}^n{L}_j^*L_j
$
(see e.g. \cite{K71, HN00}).\par
Our result directly applies to 
the study of the 
overdetermined systems of first order
partial differential operators
on $M$, that are canonically associated to
$\mathfrak{Z}(M)$
and to a $\mathbb{C}$-linear connection
on a complex vector bundle~$E\xrightarrow{\pi}M$.
\par
We also investigate the $\mathcal{C}^{\infty}$-hypoellipticity of
more general 
second order partial differential operators
on $M$, of the form 
\begin{equation}
  \label{eq:oa}
 P(u)=\sum_{j=1}^n{\bar{L}_jL_j}(u)+L_0(u)+a\,u, 
\end{equation}
where $a$ is a complex valued smooth function in 
$\mathcal{C}^{\infty}(M)$, $L_0,L_1,\hdots,L_n$ are 
smooth complex vector fields on $M$,
and only 
the imaginary part of $L_0$ is 
required to belong to the $\mathcal{C}^{\infty}(M)$-linear
span of
$L_1,\hdots,L_n,\bar{L}_1,\hdots,\bar{L}_n$.
\par
Our original motivation, and also our main applications, involve
$CR$ manifolds.
However, the study of 
\eqref{eq:ob} and of the 
operators \eqref{eq:oa} is of independent interest,
and
related questions have been considered recently in 
\cite{Ko05, christ-2005, BDKT06, pp1, pp2}.
These papers show that the consideration of complex $L_j$'s
brings completely new phenomena, 
compared with the real case
(see e.g. \cite{Hor67, ff81}).
\par
Condition \eqref{eq:be} for the subellipticity of $\mathfrak{Z}(M)$
can be viewed 
as a generalization,
at the same time, of both
the \emph{essential pseudoconcavity}
of \cite{HN00} in $CR$ geometry 
and H\"ormander's condition of \cite{Hor67} for the generalized Kolmogorov
equation.
In fact, in \S\ref{sec:C} we prove
that, at points that are \emph{generic} for
$\mathfrak{Z}(M)$ 
(in a sense made precise in Definition \ref{def:cg}), it is equivalent
to a condition \eqref{eq:cra},
that involves semidefinite 
\emph{generalized Levi forms} attached to $\mathfrak{Z}(M)$
and their kernels. This quite explicit formulation
was suggested by specific examples from \cite{AMN06}. However,
conditions \eqref{eq:be} and \eqref{eq:cra}
apply to more general contexts than $CR$ geometry.
At the beginning, we assume
neither that $\mathfrak{Z}(M)$ is a distribution of constant rank, nor
that it satisfies any formal integrability condition, nor anything special
about the intersection $\mathfrak{Z}(M)\cap\overline{\mathfrak{Z}(M)}$; but
conditions of this type need to be imposed in \S\ref{sec:C} to obtain 
the equivalence of \eqref{eq:be} and \eqref{eq:cra}. \par
The work of
\cite{Ko05, christ-2005, BDKT06} shows that the condition that
the Lie algebra generated by $\mathfrak{Z}(M)$ spans the full
complexified tangent space is 
in general not sufficient either for subellipticity
or for the $\mathcal{C}^{\infty}$-hypoellipticity of the sum of squares.
The present work is complementary, inasmuch as our sharpest results hold
away from singularities and in the case where
$\mathfrak{Z}(M)$ is a Lie algebra.\par
We reduce the question of the subellipticity of $\mathfrak{Z}(M)$
to one involving \emph{real} vector fields.
Indeed, our task is to find all real
vector fields $X$ that are \emph{enthralled} by $\mathfrak{Z}(M)$, i.e.
satisfy
\begin{equation}
  \label{eq:oc}
 \|X(u)\|_{\epsilon-1}^2\leq{C}\left(\sum_{j=1}^n\|L_j(u)\|_0^2+\|u\|_0^2\right),
\quad\forall{u}\in\mathcal{C}^{\infty}_0(U) 
\end{equation}
for some $L_1,\hdots,L_n\in\mathfrak{Z}(M)$ and
$\epsilon>0$, ${C}>0$ depending on $X$ and $U$.
In \S\ref{sec:a}
we observe that this set is a module over the Lie algebra
$\mathfrak{A}_{\mathfrak{Z}}(M)$
generated by those real vector fields $X$ for which 
\eqref{eq:oc} holds with $\epsilon=1$. This was essentially shown in \cite{K71}.
By an argument of J.J.Kohn
(see \cite{Ko05} and Lemma \ref{lm:Ae} below), we know 
that $(\mathrm{Re}\,Z)$ satisfies \eqref{eq:oc} with $\epsilon=\frac{1}{2}$
for all $Z\in\mathfrak{Z}(M)$. On the other hand,
on an open dense subset $M'$ of $M$ that excludes some \textit{singular}
points for $\mathfrak{Z}(M)$,
one can check that a real vector field $X$
satisfying \eqref{eq:oc} with $\epsilon>\frac{1}{2}$
is  necessarily the real part of some
$Z\in\mathfrak{Z}(M)$. 
Hence we conjecture that
the real vector fields $X$, satisfying \eqref{eq:oc}, coincide
on a dense open subset of $M$ with the elements of the
$\mathfrak{A}_{\mathfrak{Z}}(M)$-module generated by the real parts of
the elements of $\mathfrak{Z}(M)$. In \S\ref{sec:C} we characterize,
outside the singular set, the Lie algebra $\mathfrak{A}_{\mathfrak{Z}}(M)$.
Then equality \eqref{eq:Bcbb}, 
together with Lemma \ref{lm:Ae},
would give a complete and explicit description of the
set of real vector fields enthralled by $\mathfrak{Z}(M)$.\par
A motivation for \cite{HN00} was to understand the structure of
a large class of $CR$ 
manifolds, of higher $CR$ codimension, 
that are of finite type in the sense of \cite{BG77}, 
and are not pseudoconcave
in the sense of \cite{HN96}, because of the vanishing of their
scalar Levi forms
in some characteristic codirections.
To prove also in this case
a subelliptic estimate
for the tangential Cauchy-Riemann operator on functions, 
two of the authors
invented in \cite{HN00} 
the notion of \emph{essential pseudoconcavity}.
Compared with the
usual more restrictive notion of pseudoconcavity, 
it allows 
some of the scalar Levi forms to be zero. The subelliptic estimate
is then obtained for
weaker Sobolev norms than the $\frac{1}{2}$-norm
in \cite{HN96}. 
On essentially pseudoconcave almost $CR$ manifolds 
the maximum modulus principle and the weak unique continuation principle
for $CR$ functions
(\cite{HN00, HN03}) are valid. Under the additional assumption of 
formal integrability of the
$CR$ structure, it was also possible
to prove, in case the $CR$ manifold is
either compact or real analytic, 
finiteness and vanishing theorems for
the highest cohomology groups of the 
$\bar\partial_M$-complexes (\cite{N07}).\par
Our new condition is more general than
the weak pseudoconcavity of \cite{HN00},
as here we allow 
the scalar Levi forms in some characteristic codirections to be
semidefinite. As a heuristic motivation, consider a
$CR$ submanifold $M$ of a complex manifold $F$.
If $M$ 
is contained in a strictly pseudoconvex real hypersurface of $F$, then 
it is easy to find $L^2_{\mathrm{loc}}$
germs of $CR$ functions that are not smooth.
We can expect,
vice versa, that all  $CR$ distributions on $M$ 
are smooth when 
all germs of hypersurfaces through $M$ in  $F$ 
are strictly pseudoconcave. This is indeed the case when
our condition \eqref{eq:cra} is  verified. It is discussed in 
\S\ref{sec:examp}.
\par
As in \cite{HN00}, the homogeneous examples have strongly inspired
our investigation. After considering general homogeneous $CR$ manifolds in
\S\ref{sec:hom},
in \S\ref{sec:orb} 
we  classify
all minimal orbits of real forms $\mathbf{G}$ of
$\mathbf{G}^{\mathbb{C}}$-homogeneous
flag manifolds (see e.g. \cite{Wolf69, AMN06, AMN07})
that enjoy condition \eqref{eq:be}.
In \cite[\S 13]{AMN06}, 
together with Prof. Medori, two of the authors gave
the complete classification of the essentially pseudoconcave
minimal orbits. Here
we show that those satisfying \eqref{eq:be}
form a strictly larger class of $CR$ manifolds, on which the local 
$CR$ distributions are smooth.
\par
We collected our results on hypoellipticity in \S\ref{sec:ope}. 
Given a complex vector bundle $E\xrightarrow{\pi}M$, endowed with a
$\mathbb{C}$-linear connection $\nabla$, for each complex vector field
$Z$ on $M$ we obtain a differential operator $\nabla_Z$, acting on the
sections of $E$. We prove that the subellipticity of $\mathfrak{Z}(M)$
implies the $\mathcal{C}^{\infty}$-hypoellipticity 
of $(\nabla_Z)_{Z\in\mathfrak{Z}(M)}$.
This yields, for a compact $M$, finite dimensionality of the space of global solutions
of $(\nabla_Z(u))_{Z\in\mathfrak{Z}(M)}=0$, and closedness of the range
of $(\nabla_Z)_{Z\in\mathfrak{Z}(M)}$. For $CR$ manifolds, this implies that
the cohomology groups $H^{p,1}_{\bar\partial_M}(M)$ of the tangential 
Cauchy-Riemann complex are Hausdorff.\par
The subellipticity of $\mathfrak{Z}(M)$ gives 
$\mathcal{C}^{\infty}$-hypoellipticity for the sum of squares and also for
slightly more general operators (see Theorem \ref{tm:ell}). The question
of generalizing the Kolmogorov equation to the case of complex vector fields
is more complicated. We obtained two different formulations in
Theorems \ref{tm:cpd} and \ref{tm:bh}. In the former, we obtain
$\mathcal{C}^{\infty}$-hypoellipticity under the condition that the
$\mathfrak{A}_{\mathfrak{Z}}(M)$-module generated by the real parts of
the vector fields in $\mathfrak{Z}(M)$ and the ``time'' vector field
span the whole tangent space at a point. This distinction between ``space''
and ``time'' vector fields makes this result weaker than the one of
\cite{Hor67} for the case $\mathfrak{Z}(M)=\overline{\mathfrak{Z}(M)}$.
This is the reason we prove in Theorem \ref{tm:bh}
a result generalizing \cite{Hor67}, under the condition that the real parts
of the vector fields in $\mathfrak{Z}(M)$ satisfy \eqref{eq:oc} with
$\epsilon=1$.

\section{Subelliptic systems of complex vector fields}\label{sec:a}
Let $M$ be a smooth \textit{real} manifold 
of dimension $m$. 
For $U$ open in $M$, we denote by
$\mathcal{C}^{\infty}(U)$ (resp.
$\mathcal{C}^{\infty}_0(U)$) the space of \textit{complex
valued} smooth functions (resp. with
compact support) in $U$, 
and by  $\mathfrak{X}(U)$ (resp. $\mathfrak{X}^{\mathbb{C}}(U)$)
the Lie algebra of the smooth real (resp. complex) vector fields in~$U$.
\begin{defn}
A \emph{distribution of complex vector fields} 
$\mathfrak{Z}(M)$ in $M$ is 
any  $\mathcal{C}^{\infty}(M)$-left submodule of 
$\mathfrak{X}^{\mathbb{C}}(M)$. This means that $Z_1+Z_2$ and
$fZ$ belong to $\mathfrak{Z}(M)$ if $Z_1,Z_2,Z\in\mathfrak{Z}(M)$
and $f\in\mathcal{C}^{\infty}(M)$. \par
For each point $p\in{M}$ we set
\begin{equation}
  \label{eq:Aa}
  Z_pM=\{Z(p)\mid Z\in\mathfrak{Z}(M)\}\subset{T}^{\mathbb{C}}_pM=
\mathbb{C}\otimes_{\mathbb{R}}T_pM.
\end{equation}
The dimension of the complex vector space $Z_pM$ is called the
\emph{rank} of $\mathfrak{Z}(M)$ at $p$.
We \emph{do not} assume 
in the definition that $\mathfrak{Z}(M)$ has a constant rank.
The points where the rank of $\mathfrak{Z}(M)$ is not constant
on a neighborhood are the \emph{singularities} of $\mathfrak{Z}(M)$.
If $\mathfrak{Z}(M)=\{Z\in\mathfrak{X}^{\mathbb{C}}(M)\mid Z(p)\in{Z}_pM,
\;\forall{p}\in{M}\}$, we say that $\mathfrak{Z}(M)$ has at most
simple singularities. 
\par
We say that $\mathfrak{Z}(M)$ is \emph{formally integrable} if 
\begin{equation}
  \label{eq:Bbb}
  [\mathfrak{Z}(M),\mathfrak{Z}(M)]\subset\mathfrak{Z}(M).
\end{equation}
\par
Distributions of \emph{real} vector fields 
and their singular points are
defined likewise.
\end{defn}
\begin{defn}
The distribution $\mathfrak{Z}(M)$ is said to be \emph{subelliptic} at $p\in{M}$
if there is an open neighborhood $U$ of $p$ in $M$,
a real $\epsilon>0$, a constant $C>0$,
and a finite set
$L_1,\hdots,L_n$ of vector fields from $\mathfrak{Z}(M)$, such that
\eqref{eq:ob} is satisfied,
where $\|\;\cdot\;\|_{\epsilon}$ and $\|\;\cdot\;\|_0$ are the
$\epsilon$-Sobolev norms and the $L^2$-norm with respect to some
Riemannian metric on $M$ (see e.g. \cite{heb96}).
\end{defn}
\begin{rmk}
When $\mathfrak{Z}(M)$ is the complexification of a distribution of
\emph{real} vector fields $\mathfrak{Y}(M)\subset\mathfrak{X}(M)$,
then \eqref{eq:ob}, at a generic point $p\in{M}$, 
is equivalent to the fact
that $\mathfrak{Z}(M)$ and its 
higher order commutators span the
whole complexified
tangent space $T_{p}^{\mathbb{C}}M$ 
(see, e.g. \cite{Hor67, ff81}). However this condition is neither
necessary nor 
sufficient, and does not imply 
$\mathcal{C}^{\infty}$-hypoellipticity of the associated sum of squares operators
when the
vector fields are complex
(cf. \cite{Ko05,christ-2005,BDKT06, pp2}).
\end{rmk}
\begin{defn}\label{def:Ac}
 We say that $\mathfrak{Z}(M)$ \emph{enthralls} a vector field
$Z\in\mathfrak{X}^{\mathbb{C}}(M)$ if
\begin{equation}
  \label{eq:Ac}
  \left\{
    \begin{gathered} \forall 
      {U}^{\mathrm{open}}\Subset{M},\;\exists{L_1,\hdots,L_n}\in
\mathfrak{Z}(M),\;\exists\epsilon>0,\,C>0\;\text{s.t.}\\
\|Z(u)\|_{\epsilon-1}^2\leq{C}\left(\sum_{j=1}^n\|L_j(u)\|_0^2+\|u\|_0^2\right)
,\quad\forall u\in\mathcal{C}^{\infty}_0(U).
    \end{gathered}\right.
\end{equation}
Set
\begin{align}
  \label{eq:Ad}
\mathbb{S}_{\mathfrak{Z}}(M)
&=\{Z\in\mathfrak{X}^{\mathbb{C}}(M)
\mid \mathfrak{Z}(M)\;\text{enthralls}\; Z\}.
\end{align}
\end{defn}
We notice that
$\mathbb{S}_{\mathfrak{Z}}(M)$ is a distribution of complex
vector fields containing $\mathfrak{Z}(M)$, 
and $\mathbb{S}_{\mathfrak{Z}}(M)\cap\mathfrak{X}(M)$ is a distribution of real vector fields,
and that both are spaces of global sections of fine sheaves of 
left modules, the first over the sheaf of 
germs of \textit{complex} valued smooth functions, the second over the
sheaf of  
germs of \textit{real} valued smooth functions on~$M$.
\par
By the real Frobenius theorem we obtain (see e.g. \cite{Hor67})
\begin{prop}
Let $\mathfrak{W}_{\mathfrak{Z}}(M)$ be the Lie subalgebra 
of $\mathfrak{X}^{\mathbb{C}}(M)$ generated by
$\mathfrak{Z}(M)+\overline{\mathfrak{Z}(M)}$. Then, if
$\mathfrak{W}_{\mathfrak{Z}}(M)$ is a distribution of constant rank on $M$,
we have
\begin{equation}
  \label{eq:Aee}
 \mathbb{S}_{\mathfrak{Z}}(M)\cap\mathfrak{X}(M)\subset 
\mathbb{S}_{\mathfrak{Z}}(M)\subset\mathfrak{W}_{\mathfrak{Z}}(M).
\end{equation}
\end{prop}
\begin{rmk}
  The complex distribution $\mathfrak{Z}(M)$ is subelliptic at $p\in{M}$
if and only if
\begin{equation}
  \label{eq:Aof}
  \{X(p)\mid X\in\mathbb{S}_{\mathfrak{Z}}(M)\cap\mathfrak{X}(M)\}=T_pM.
\end{equation}
\end{rmk}
\begin{rmk}
If $Z\in\mathbb{S}_{\mathfrak{Z}}(M)$ and
$\Lambda_{-1}$ is a properly supported pseudodifferential operator 
of order $(-1)$ having its symbol in the class $S^{-1}_{1,0}$, 
then
$\Lambda_{-1}\circ{Z}$ is a \emph{subelliptic multiplier} for
$\mathfrak{Z}(M)$ in the sense of J.J.Kohn (see e.g. \cite{Ko03}).
\end{rmk}
Pseudodifferential operators will be an important tool 
in the following.
For their definition and properties we refer to
\cite[Chapter XVIII]{Hor85}.
\par
If $U$ is an open subset of $M$,
 and $s$ a real number,
 we shall denote by
$\Psi^{s}(U)$ the space of properly supported pseudodifferential operators
in $U$, of order less or equal than $s$, with 
symbol in
$S^{s}_{1,0}$. For each coordinate neighborhood $V\subset{U}$,
a $\Lambda\in\Psi^s(U)$ is defined, in the local coordinates, by
\begin{equation}
  \label{eq:Ah}
  \Lambda(u)=\iint_{U\times\mathbb{R}^m}e^{i\langle{x-y},\xi\rangle}
a(x,\xi)u(y)dyd\xi,\quad\text{for}\; u\in\mathcal{C}^{\infty}_0(V),
\end{equation}
with $a\in\mathcal{C}^{\infty}(V\times\mathbb{R}^m)$ and
\begin{equation}
  \label{eq:Ai}
 \left\{ \begin{gathered}
 \forall K\Subset{V},\;\forall\alpha,\beta\in\mathbb{N}^m,\;\exists
C=C(K,\alpha,\beta)>0\;\text{s.t.}\\
\left|D_x^{\alpha}D_\xi^{\beta}a(x,\xi)\right|\leq
C\,(1+|\xi|)^{s-|\beta|},\;\forall (x,\xi)\in{K}\times\mathbb{R}^m.
  \end{gathered}\right.
\end{equation}
The fact that $\Lambda\in\Psi^s(U)$ is \textit{properly supported} means that
for every $K\Subset{U}$ there is $K'\Subset{U}$ such that
\begin{equation}
  \label{eq:Aj}
  \mathrm{supp}(u)\subset{K}\Rightarrow\mathrm{supp}(\Lambda(u))\subset{K}'.
\end{equation}
The following Lemma is essentially contained in \cite[p.949]{Ko05}.
\begin{lem}\label{lm:Ae}
Let $Z\in\mathfrak{X}^{\mathbb{C}}(M)$ be any complex vector field.
For every relatively compact open subset $U\Subset{M}$ there is a constant
$C>0$ such that 
\begin{equation}
  \label{eq:Aaa}
  \|\bar{Z}(u)\|^2_{-\frac{1}{2}}\leq C\left(\|Z(u)\|^2_0+\|u\|_0^2\right)
\quad\forall{u}\in\mathcal{C}^{\infty}_0(U).
\end{equation}  
Hence
\begin{gather}
  \label{eq:Aab}
  \mathfrak{Z}(M)+\overline{\mathfrak{Z}(M)}\subset
\mathbb{S}_{\mathfrak{Z}}(M),\\ \label{equ:Aab}
\{Z+\bar{Z}\mid Z\in{\mathfrak{Z}(M)}\}\subset
\mathbb{S}_{\mathfrak{Z}}(M)\cap\mathfrak{X}(M).
\end{gather}
\end{lem}
\begin{proof}
  Let $Z\in\mathfrak{Z}(M)$ and
$U^{\mathrm{open}}\Subset{M}$. Then, with $\Lambda_0\in\Psi^0(U)$ and
constants $C_1,C_2>0$, we obtain
\begin{equation*}
  \begin{aligned}
    \|\bar{Z}(u)\|_{-\frac{1}{2}}^2=(\bar{Z}(u)|\Lambda_0(u))_0
&\leq \left|(\Lambda_0^*(u)|Z(u))_0\right|+C_1\|u\|_0^2\\
&\leq C_2\left(\|Z(u)\|_0+\|u\|_0\right)\|u\|_0,\quad
\forall u\in\mathcal{C}^{\infty}(U).
  \end{aligned}
\end{equation*}
This yields \eqref{eq:Aaa}, and hence also \eqref{eq:Aab} and
\eqref{equ:Aab}.
\end{proof}

In particular, if $\mathfrak{Z}(M)\neq\{0\}$,
the distribution $\mathbb{S}_{\mathfrak{Z}}(M)\cap\mathfrak{X}(M)$ 
of \textit{real} vector fields that are enthralled by
$\mathfrak{Z}(M)$ is not trivial. 
\begin{defn}\label{def:Aof}
We denote by $\mathbb{E}_{\mathfrak{Z}}(M)$ 
the set of $Z\in\mathfrak{X}^{\mathbb{C}}(M)$
such that
\begin{equation}
  \label{eq:Aac}
  \left\{
    \begin{gathered}
      \forall {U}^{\mathrm{open}}\Subset{M},\;
\exists L_1,\hdots,L_n\in\mathfrak{Z}(M),\;\exists
C>0\;\text{s.t}\\
\|Z(u)\|_0^2\leq{C}\left(\sum_{j=1}^n\|L_j(u)\|_0^2+\|u\|_0^2\right),\quad
\forall u\in\mathcal{C}^{\infty}_0(U).
    \end{gathered}
\right.
\end{equation}
\end{defn}
Also $\mathbb{E}_{\mathfrak{Z}}(M)$ is a distribution 
of complex vector fields, 
with 
\begin{eqnarray}\label{eq:Aaf}
 \mathfrak{Z}(M)\subset\mathbb{E}_{\mathfrak{Z}}(M)
\subset\mathbb{S}_{\mathfrak{Z}}(M). 
\end{eqnarray}
As a consequence of Lemma \ref{lm:Ae}, we get
\begin{lem}\label{Lem:Aog}
  \begin{equation}
    \label{eq:Aot}
    \mathbb{E}_{\mathfrak{Z}}(M)+\overline{\mathbb{E}_{\mathfrak{Z}}(M)}
\subset\mathbb{S}_{\mathfrak{Z}}(M).
  \end{equation}
\end{lem}
\begin{defn}\label{def:hcg} 
Set
\begin{align}
  \label{eq:bb}
  \mathbb{A}_{\mathfrak{Z}}(M)&=\mathbb{E}_{\mathfrak{Z}}(M)\cap\mathfrak{X}(M),
\\[4pt] \label{eq:bbt}
\mathfrak{A}_{\mathfrak{Z}}(M)&=\text{the Lie subalgebra of $\mathfrak{X}(M)$
generated by $\mathbb{A}_{\mathfrak{Z}}(M)$},
\\[4pt]
\label{eq:bc}
\mathbb{T}^{(0)}_{\mathfrak{Z}}(M)&=\{Z+\bar{Z}\mid Z\in
{\mathfrak{Z}}(M)\},\\[4pt]
\label{eq:bd}
\mathbb{T}^{(h)}_{\mathfrak{Z}}(M)&=\left\langle
[X,Y]\mid X\in\mathbb{A}_{\mathfrak{Z}}(M),\; 
Y\in\mathbb{T}^{(h-1)}_{\mathfrak{Z}}(M)\}
\right\rangle,\quad
\text{for $h\geq{1}$,}
\\ \label{eq:bdt}
\mathfrak{T}_{\mathfrak{Z}}(M)
&=\sum_{h=0}^{\infty}\mathbb{T}^{(h)}_{\mathfrak{Z}}(M).
\end{align}
We shall consider the condition at $p\in{M}$: 
\begin{equation}
  \label{eq:be}
  \{X(p)\mid X\in\mathfrak{T}_{\mathfrak{Z}}(M)\}=T_{p}M.
\end{equation}
\end{defn}
\begin{rmk}
If condition \eqref{eq:be} is satisfied at a point
$p_0\in{M}$, then 
it is also satisfied at all points $p$ in an open neighborhood
$U$ of $p_0$.
\end{rmk}
It is convenient to introduce the notation $[Z_1,\hdots,Z_m]$
for the \textit{higher order commutator} of smooth 
real or complex vector
fields. It is recursively defined by
\begin{equation}
  \label{eq:hcm}
  \begin{cases}
[Z_1]=Z_1,\\
[Z_1,Z_2]=Z_1Z_2-
Z_2Z_1,\\
[Z_1,Z_2,\hdots,Z_m]=[Z_1,[Z_2,\hdots,Z_m]] &\text{for}\; m\geq{3}.
\end{cases}
\end{equation}
\begin{prop} \label{pp:Abb}
The distribution of real vector fields 
$\mathfrak{T}_{\mathfrak{Z}}(M)$ is an
$\mathfrak{A}_{\mathfrak{Z}}(M)$-Lie-module.  
\end{prop}
\begin{proof}
We prove by recurrence on $r\geq{1}$ that, if 
$X_1,\hdots,X_r\in\mathbb{A}_{\mathfrak{Z}}(M)$ and 
$Y\in\mathfrak{T}_{\mathfrak{Z}}(M)$, then
also $[[X_1,\hdots,X_r],Y]\in\mathfrak{T}_{\mathfrak{Z}}(M)$. This follows from
the definition of $\mathfrak{T}_{\mathfrak{Z}}(M)$ for $r=1$. Assume now that
$r>1$ and that 
$\mathfrak{T}_{\mathfrak{Z}}(M)$ 
contains all commutators $[[X_1,\hdots,X_{r-1}],Y]$
with $X_1,\hdots,X_{r-1}\in\mathbb{A}_{\mathfrak{Z}}(M)$ and 
$Y\in\mathfrak{T}_{\mathfrak{Z}}(M)$.
If $X_1,\hdots,X_r\in\mathbb{A}_{\mathfrak{Z}}(M)$ and 
$Y\in\mathfrak{T}_{\mathfrak{Z}}(M)$,
we obtain
\begin{equation*}
[[X_1,\hdots,X_r],Y]=
-[[X_2,\hdots,X_r],[X_1,Y]]+[X_1,[[X_2,\hdots,X_r],Y]].
\end{equation*}
Since $[X_1,Y]\in\mathfrak{T}_{\mathfrak{Z}}(M)$,
by our inductive assumption the first summand on the right hand side
also belongs to $\mathfrak{T}_{\mathfrak{Z}}(M)$. 
By the inductive assumption, the
commutator
$[[X_2,\hdots,X_r],Y]$ belongs to $\mathfrak{T}_{\mathfrak{Z}}(M)$, 
and hence also the second
summand in the right hand side belongs to $\mathfrak{T}_{\mathfrak{Z}}(M)$.
The proof is complete.
\end{proof}
\begin{prop}\label{pp:bk}
The Lie algebra of real vector fields $\mathfrak{A}_{\mathfrak{Z}}(M)$
is contained in $\mathbb{S}_{\mathfrak{Z}}(M)\cap\mathfrak{X}(M)$, and 
$\mathbb{S}_{\mathfrak{Z}}(M)\cap\mathfrak{X}(M)$ is 
an $\mathfrak{A}_{\mathfrak{Z}}(M)$-Lie-submodule of $\mathfrak{X}(M)$.
In particular, we have the inclusion
\begin{equation}
  \label{eq:nbk}
  \mathfrak{T}_{\mathfrak{Z}}(M)\subset\mathbb{S}_{\mathfrak{Z}}(M)
\cap\mathfrak{X}(M).
\end{equation}
\end{prop}
\begin{proof}
Let $U$ be a relatively compact open subset of $M$. Assume that
$X\in\mathbb{A}_{\mathfrak{Z}}(M)$,
$Y\in\mathbb{S}_{\mathfrak{Z}}(M)\cap\mathfrak{X}(M)$ and 
let $\epsilon>0$ be such that
that $Y=Z$
satisfies the estimate in \eqref{eq:Ac}. We can assume that
$0<\epsilon\leq\frac{1}{2}$.
If $U'$ is an open relatively compact subset of $U$, 
we have, with some~$\Lambda_{\epsilon-1}\in\Psi^{\epsilon-1}(U)$
and suitable positive constants $\mathrm{C}_0$, $\mathrm{C}_1$,
\begin{equation*}
  \begin{aligned}
    \|[X,Y](u)\|_{\frac{\epsilon}{2}-1}^2&\leq
\left|([X,Y](u)|\Lambda_{\epsilon-1}(u))_0\right|\\
&\leq\left|(XY(u)|\Lambda_{\epsilon-1}(u))_0\right|+
\left|(YX(u)|\Lambda_{\epsilon-1}(u))_0\right|\\
&\leq \left|(\Lambda_{\epsilon-1}^*(Y(u))|X(u))_0\right|+
\left|(X(u)|\Lambda_{\epsilon-1}Y(u))\right|\\
&\qquad\qquad +\mathrm{C}_0\|u\|\,\left(\|u\|_0+\|Y(u)\|_{\epsilon-1}+
\|X(u)\|_0\right)\\
&\leq \mathrm{C}_1\left(\|X(u)\|_0^2+\|Y(u)\|^2_{\epsilon-1}+
\|u\|_0^2\right),\qquad\quad
\forall u\in\mathcal{C}^{\infty}(U')
  \end{aligned}
\end{equation*}
The last term of this chain of inequalities is bounded by
a constant times $\left(\sum_{j=1}^n\|L_j(u)\|^2+\|u\|_0^2\right)$,
for a suitable choice of $L_1,\hdots,L_n\in\mathfrak{Z}(M)$.
This shows that $[X,Y]\in\mathbb{S}_{\mathfrak{Z}}(M)\cap\mathfrak{X}(M)$.
\par
Since $\mathbb{A}_{\mathfrak{Z}}(M)\subset\mathbb{S}_{\mathfrak{Z}}(M)\cap\mathfrak{X}(M)$,
also $\mathfrak{A}_{\mathfrak{Z}}(M)\subset\mathbb{S}_{\mathfrak{Z}}(M)
\cap\mathfrak{X}(M)$.
The argument in the proof of Proposition \ref{pp:Abb}
shows that, since $[\mathbb{A}_{\mathfrak{Z}}(M),\mathbb{S}_{\mathfrak{Z}}(M)
\cap\mathfrak{X}(M)]
\subset \mathbb{S}_{\mathfrak{Z}}(M)\cap\mathfrak{X}(M)$, this distribution 
is
an $\mathfrak{A}_{\mathfrak{Z}}(M)$-Lie-submodule of $\mathfrak{X}(M)$.
Then the inclusion \eqref{eq:nbk} is a consequence of the inclusion
\eqref{equ:Aab}.
\end{proof}
By using Proposition \ref{pp:bk} we obtain 
\begin{cor}\label{cor:sa}
Let $\mathfrak{Z}(M)$ be a smooth distribution of complex vector fields
on $M$. Then $\mathfrak{Z}(M)$ is subelliptic at all points $p\in{M}$
at which condition \eqref{eq:be} of Definition \ref{def:hcg}
is satisfied.
\end{cor}
Corollary \ref{cor:sa} is a trivial consequence of 
the inclusion \eqref{eq:nbk}. 
However, in \S\ref{sec:C} we will show
that actually we are able, in case $\mathfrak{Z}(M)$ is formally integrable,
to compute 
explicitly the left hand side of \eqref{eq:be} at 
the points of an open dense subset of $M$, where
$\mathfrak{Z}(M)$ satisfies some genericity assumptions.
\begin{rmk} \label{rmk:An}
When $\mathfrak{Z}(M)$ is 
the complexification of a distribution of \textit{real} vector fields,
then $\mathfrak{A}_{\mathfrak{Z}}(M)=\mathfrak{T}_{\mathfrak{Z}}(M)$, and
condition \eqref{eq:be} is equivalent to H\"ormander's condition
in \cite{Hor67}. 
Thus Corollary \ref{cor:sa} can be viewed as a generalization
of the analogous result for distributions of real vector fields.
\end{rmk}
\begin{rmk}
If $\mathfrak{Z}(M)$ is the distribution of $(0,1)$-vector fields
of an almost $CR$ manifold $M$, 
it follows from \S\ref{sec:C} below that the 
\textit{essential pseudoconcavity} condition of
\cite{HN00} implies 
that $\mathbb{E}_{\mathfrak{Z}}(M)=\mathfrak{Z}(M)+\overline{\mathfrak{Z}(M)}$ 
and
that $M$ is of finite type in the sense of \cite{BG77}. Therefore,
Corollary \ref{cor:sa} also generalizes \cite[Theorem 4.1]{HN00}.
\end{rmk}
\section{The distributions  \texorpdfstring{$\mathbb{K}_{\mathfrak{Z}}(M)$ and 
$\Theta_{\mathfrak{Z}}(M)$}{K Z (M) and Theta Z (M)}} \label{sec:C}
As pointed out after the statement of Corollary \ref{cor:sa},
condition \eqref{eq:be} becomes an effective criterion for subellipticity
when it is possible to give an explicit description of
$\mathbb{E}_{\mathfrak{Z}}(M)$, or of some nontrivial part of it.
We begin by giving an upper bound for $\mathbb{E}_{\mathfrak{Z}}(M)$.
\begin{lem} \label{lem:Aoh} 
Let $\mathfrak{Z}(M)$ be a distribution of complex vector fields.\par
If $(\mathfrak{Z}(M)+\overline{\mathfrak{Z}(M)})$
has at most simple singularities, and in particular if 
$(\mathfrak{Z}(M)+\overline{\mathfrak{Z}(M)})$
has constant rank, then
\begin{equation}
  \label{eq:Aad}
  \mathbb{E}_{\mathfrak{Z}}(M)\subset\mathfrak{Z}(M)+\overline{
\mathfrak{Z}(M)}.
\end{equation}
\end{lem}
\begin{proof} 
Since $\mathbb{E}_{\mathfrak{Z}}(M)\subset
\mathbb{E}_{\mathfrak{Z}+\overline{\mathfrak{Z}}}(M)$,
we can reduce the proof to the case where
$\mathfrak{Z}(M)=\mathfrak{Z}(M)+\overline{
\mathfrak{Z}(M)}$ is the complexification of a distribution of real
vector fields and  $Z\in\mathbb{E}_{\mathfrak{Z}}(M)$ is real.
Fix $p\in{M}$ and take a coordinate patch $U$ of $p$ in $M$,
centered at $p$,
for which
\eqref{eq:Aac} 
holds, for real
$L_1,\hdots,L_n\in\mathfrak{X}(M)\cap\mathfrak{Z}(M)$ such that
$L_1(p), \hdots, L_n(p)$
generate $Z_pM$.
We apply the inequality in \eqref{eq:Aac} to
the test function 
${u}_{\tau}(x)=\tau^{\frac{m-4}{4}}\chi(x)e^{i\tau\langle{x},\xi\rangle-(\tau/2) |x|^2}$,
where $\chi(x)\in\mathcal{C}^{\infty}_0(U)$ satisfies $\chi(x)=1$
for $x$ in a neighborhood of $0$.
 Denote by $z(x,\xi)$ and $\ell_j(x,\xi)$ the
symbols of $Z$ and $L_j$, respectively. We obtain
\begin{equation*}
L_j(u_{\tau})=\tau^{\frac{m-4}{4}}\left(\tau\ell_j(x,\xi+ix)+L_j(\chi)\right)
e^{i\tau\langle{x},\xi\rangle-(\tau/2) |x|^2}.
\end{equation*}
Computing the integral by the change of variables $y=x\sqrt{\tau}$,
we obtain
\begin{equation*}
  \|L_j(u_\tau)\|_0^2=\int{\chi^2(y/\sqrt{\tau})|\ell_j(y/\sqrt{\tau},i\xi+
y/\sqrt{\tau})|^2e^{-|y|^2}dy}+{O}(\tau^{-\infty}).
\end{equation*}
Likewise, we have
\begin{equation*}
  \|Z(u_\tau)\|_0^2=\int{\chi^2(y/\sqrt{\tau})|z(y/\sqrt{\tau},i\xi+
y/\sqrt{\tau})|^2e^{-|y|^2}dy}+{O}(\tau^{-\infty}).
\end{equation*}
By letting $\tau$ tend to $\infty$ in the estimate 
\eqref{eq:Aac}, we obtain that
\begin{equation*}
  |z(0,\xi)|^2\leq{C}\sum_{j=1}^n|\ell_j(0,\xi)|^2,\quad
\forall\xi\in\mathbb{R}^m.
\end{equation*}
Since $L_1,\hdots,L_n$ are real,
the above inequality implies that $Z(p)\in{Z}_pM$.
Since $p$ was an arbitrary point of $M$,
this implies that $Z\in\mathfrak{Z}(M)$.
\end{proof}
\begin{rmk}
The proof of Lemma \ref{lem:Aoh} yields a stronger statement:\par
\textit{
Let $\mathfrak{Z}(M)$ be a distribution of complex vector fields and
assume that $\mathfrak{Z}(M)+
\overline{\mathfrak{Z}(M)}$ has at most simple singularities.
If $Z\in\mathbb{S}_{\mathfrak{Z}}(M)$ and}
\begin{equation}
  \label{eq:CAb}
\left\{  \begin{gathered}
    \forall U^{\mathrm{open}}\Subset{M},\quad\exists \epsilon>\frac{1}{2},\;
\exists L_1,\hdots,L_n\in\mathfrak{Z}(M),\;
\exists C>0\;\text{s.t.}\\
\|Z(u)\|_{\epsilon-1}^2\leq C\left(\sum_{j=1}^n\|L_j(u)\|_0^2+\|u\|_0^2\right),\quad
\forall u\in\mathcal{C}^{\infty}_0(U),
  \end{gathered}\right.
\end{equation}
\textit{
then $Z\in\mathfrak{Z}(M)+
\overline{\mathfrak{Z}(M)}$.}\par
   It suffices indeed to apply \eqref{eq:CAb}, in a coordinate patch $U$,
as in the proof of Lemma \ref{lem:Aoh}, to the test functions
$v_{\tau}=\tau^{{(1-\epsilon')}/{2}}u_{\tau}$, 
with $\frac{1}{2}<\epsilon'<\epsilon$,
and let $\tau\to+\infty$. Then, if $\ell_j(0,\xi)=0$ for $j=1,\hdots,n$,
the right hand side of \eqref{eq:CAb}
stays bounded, while the left hand side tends to $+\infty$, unless
$z(p,\xi)=0$. 
\end{rmk}
\subsection{The distribution $\Theta_{\mathfrak{Z}}(M)$}
Lemma \ref{lem:Aoh} suggests that, in order to find non trivial elements
of $\mathbb{E}_{\mathfrak{Z}}(M)$, one should  
search in
$\overline{\mathfrak{Z}(M)}$. 
To this aim,
we introduce the following
\begin{defn} \label{def:Ag} Given a distribution
$\mathfrak{Z}(M)$ of complex vector fields, we set
\begin{equation}\label{eq:nd}
 \Theta_{\mathfrak{Z}}(M)=\left\{Z\in\mathfrak{Z}(M)\left|{\begin{matrix}
\exists r\geq0,\; \exists Z_1,\hdots,Z_r\in\mathfrak{Z}(M),\;
\text{s.t.}\\
[Z,\bar{Z}]+\sum_{j=1}^r[Z_j,\bar{Z}_j]
\in\mathfrak{Z}(M)+\overline{\mathfrak{Z}(M)}
\end{matrix}}\right\}\right. .
\end{equation}
\end{defn}
\begin{lem}\label{lm:hlb}
The set $\Theta_{\mathfrak{Z}}(M)$ is a distribution of 
complex vector fields.
\end{lem}
\begin{proof}
Clearly, if $Z\in\Theta_{\mathfrak{Z}}(M)$ and $\phi\in\mathcal{E}(M)$,
the product $\phi\,Z$ also belongs to $\Theta_{\mathfrak{Z}}(M)$.
To prove that $\Theta_{\mathfrak{Z}}(M)$ contains the 
finite sums of its elements,
it suffices to show that,
if, for some $r\geq{1}$, 
 $Z_0,\hdots,Z_r\in\Theta_{\mathfrak{Z}}(M)$ and
$\sum_{j=0}^r[Z_j,\bar{Z}_j]\in\mathfrak{Z}(M)+\overline{\mathfrak{Z}(M)}$,
then also $Z_0+Z_1\in\Theta_{\mathfrak{Z}}(M)$. This follows from:
\quad
 $ [Z_0+Z_1,\overline{Z_0+Z_1}]+[Z_0-Z_1,\overline{Z_0-Z_1}]=
2\left([Z_0,\bar{Z}_0]+[Z_1,\bar{Z}_1]\right).$
\end{proof}
\begin{lem}\label{lm:ai}
Let $\mathfrak{Z}(M)$ be a distribution of complex vector fields and
let $\Theta_{\mathfrak{Z}}(M)$ be defined by \eqref{eq:nd}. 
Then
\begin{equation}
  \label{eq:nbi}
  \overline{\Theta_{\mathfrak{Z}}(M)}\subset \mathbb{E}_{\mathfrak{Z}}(M).
\end{equation}
In particular, if $Z\in\Theta_{\mathfrak{Z}}(M)$, then
$Z+\bar{Z}\in\mathbb{A}_{\mathfrak{Z}}(M)$.
\end{lem}
\begin{proof}
The proof closely follows that of \cite[Theorem 2.5]{HN00}. If
$Z\in\Theta_{\mathfrak{Z}}(M)$, by \eqref{eq:nd}, there are
$Z_1,\hdots,Z_r,Z_{r+1}\in\mathfrak{Z}(M)$ such that
\begin{equation}\label{eq:nbj}
  [Z,\bar{Z}]+\sum_{j=1}^r[Z_j,\bar{Z}_j]=Z_{r+1}-\bar{Z}_{r+1}.
\end{equation}
Set $Z_0=Z$ and let $Z_j^*=-\bar{Z}_j+a_j$, with $a_j\in\mathcal{C}^{\infty}(M)$,
the $L^2$ formal adjoint of $Z_j$, for $0\leq{j}\leq{r+1}$. 
Integrating by parts, and using \eqref{eq:nbj} to compute the 
sum of the commutators,
we obtain, for all $u\in\mathcal{C}^{\infty}_0(M)$,
\begin{equation*}
  \begin{aligned}
    \sum_{j=0}^r\|\bar{Z}_j(u)\|^2_0&=
-\sum_{j=0}^r((Z_j-\bar{a}_j)(\bar{Z}_j(u))|u)_0\\
&=-\sum_{j=0}^r\left\{([Z_j,\bar{Z}_j](u)|u)_0
+(\bar{Z}_jZ_j(u)|u)_0-(\bar{Z}_j(u)|a_ju)_0\right\}\\
&=\sum_{j=0}^r\|Z_j(u)\|_0^2+\sum_{j=0}^{r+1}\mathrm{Re}\,(Z_j(u)|b_ju)
+\mathrm{Re}(u|b'u)_0,
  \end{aligned}
\end{equation*}
where $b_j,b'\in\mathcal{C}^{\infty}(M)$. Hence 
the inequality in \eqref{eq:Aac} (with $\bar{Z}$ replacing $Z$)
follows,
with $n=r+2$ and $L_j=Z_{j-1}$ for $j=1,\hdots,r+2$.
\end{proof}
\subsection{The characteristic bundle and the scalar Levi forms}
\label{sec:pclf}
Next we define the characteristic bundle of $\mathfrak{Z}(M)$ and 
the analogues,  for a general distribution $\mathfrak{Z}(M)$,
of the scalar Levi forms of $CR$ manifolds.
\begin{defn}
The \emph{characteristic bundle} of $\mathfrak{Z}(M)$ is the
set $H^0M\subset{T}^*M$, consisting of the 
\textit{real} covectors $\xi$ with 
$\langle{L},{\xi}\rangle=0$ for all $L\in\mathfrak{Z}(M)$.\par
If the set $H^0_pM$ of characteristic covectors at $p\in{M}$
is $\{0\}$, we say that  $\mathfrak{Z}(M)$ is \emph{elliptic}
at $p$.
\end{defn}
For each $p\in{M}$, the set $H^0_pM=H^0M\cap{T}^*_pM$
is a vector space.
Its dimension $\mathrm{dim}_{\mathbb{R}}H^0_pM$ is an upper semicontinuous
function of $p\in{M}$.
In particular, if  $\mathfrak{Z}(M)$ is {elliptic} at a point $p_0\in{M}$,
it is elliptic for $p$
in an open neighborhood $U$ of $p_0$. In this
case \eqref{eq:ob} is valid with $\epsilon=1$ by G\r{a}rding's inequality.
Hence obstructions to the validity of the subelliptic estimate
\eqref{eq:ob} may come only from the characteristic codirections
of $\mathfrak{Z}(M)$. \par\smallskip

We restate condition \eqref{eq:be} in terms of the characteristic bundle.
\begin{prop}\label{prp:ca}
Condition \eqref{eq:be} at $p\in{M}$ is equivalent to  
\begin{equation}
  \label{eq:pcb}
\left\{\begin{gathered}
  \forall\xi\in{H}^0_{p}M,\;\text{with}\;\xi\neq{0},
\\
\exists Z_0\in\mathfrak{Z}(M),\,
Z_1,\hdots,Z_r\in\mathbb{E}_{\mathfrak{Z}}(M)\cap
\overline{\mathbb{E}_{\mathfrak{Z}}(M)},\; \\
s.t.\qquad i\xi([Z_1,\hdots,Z_r,\bar{Z}_0])\neq{0},
\end{gathered}\right.
\end{equation}
\end{prop}
\begin{proof}
This follows because the elements of $\mathfrak{T}_{\mathfrak{Z}}(M)$ 
can be expressed as linear combinations of the real parts of
the $[Z_1,\hdots,Z_r,\bar{Z}_0]$'s, with 
$Z_0\in\mathfrak{Z}(M)$ and 
$Z_1,\hdots,Z_r\in\mathbb{E}_{\mathfrak{Z}}(M)\cap
\overline{\mathbb{E}_{\mathfrak{Z}}(M)}$, and vice versa,
the real and imaginary parts of these $[Z_1,\hdots,Z_r,\bar{Z}_0]$
belong to $\mathfrak{T}_{\mathfrak{Z}}(M)$.
\end{proof}
\begin{defn}
If $\xi\in{H}^0_pM$, we define the \emph{scalar Levi form} 
of $\mathfrak{Z}(M)$ at $\xi$ as the Hermitian symmetric form
\begin{equation}
  \label{eq:sb}
  \mathfrak{L}_{\xi}(L_1,\bar{L}_2)=i\xi([L_1,\bar{L}_2]) \quad
\text{for}\quad L_1,L_2\in\mathfrak{Z}(M).
\end{equation}
The value of the right hand side of \eqref{eq:sb}
only depends on the values
$L_1(p)$, $L_2(p)$ of the two vector fields $L_1,L_2$
at the base point $p=\pi(\xi)$.
Thus \eqref{eq:sb} is a Hermitian symmetric form on the
finite dimensional complex vector space~$Z_pM$.
\end{defn}
\begin{rmk}
In the case where $\mathfrak{Z}(M)$ is the space of $(0,1)$-vector fields
of an abstract $CR$ manifold, it was shown
in \cite{HN96} 
that the subelliptic estimate \eqref{eq:ob}
is valid with $\epsilon=1/2$ under the assumption
that, for every 
$\xi\in{H}^0_pM\setminus\{0\}$,
the Levi form $\mathfrak{L}_{\xi}$ is indefinite; 
this assumption was weakened to allow $\mathfrak{L}_{\xi}=0$
for some nonzero characteristics $\xi$ in
\cite{HN00}. These results suggest that the
obstructions to the validity of \eqref{eq:ob}
come from the characteristic $\xi$'s for which
$\mathfrak{L}_{\xi}\neq{0}$ is semidefinite.
\end{rmk} 
\subsection{The distribution $\mathbb{K}_{\mathfrak{Z}}(M)$}
Our next aim is to relate $\mathbb{E}_{\mathfrak{Z}}(M)$ and
the \textit{Levi forms} associated to
$\mathfrak{Z}(M)$.
\begin{defn}
  Define \begin{gather}
  \label{eq:pcg}
  H^{\oplus}M=\left\{\xi\in{H}^0M\mid\mathcal{L}_{\xi}\geq{0}\right\},
\\
  \label{eq:Bba}
  \mathbb{K}_{\mathfrak{Z}}(M)=\{Z\in\mathfrak{Z}(M)\mid\mathcal{L}_{\xi}(Z,
\bar{Z})=0,\;\forall\xi\in{H}^{\oplus}M\}.
\end{gather}
\end{defn}
\begin{prop}\label{prop:Bbb}
For every distribution $\mathfrak{Z}(M)\subset\mathfrak{X}^{\mathbb{C}}(M)$,
the set $\mathbb{K}_{\mathfrak{Z}}(M)$ is also a distribution.\par
Assume in addition that $\mathfrak{Z}(M)$ is formally integrable
and that $\mathfrak{Z}(M)$, $\mathfrak{Z}(M)\cap\overline{\mathfrak{Z}(M)}$
are both distributions of constant rank. 
Then
\begin{equation}
  \label{eq:Bbc}\mathfrak{Z}(M)\cap
\overline{\mathbb{E}_{\mathfrak{Z}}(M)}
\subset
\mathbb{K}_{\mathfrak{Z}}(M).
\end{equation}
\end{prop}
\begin{proof} The first claim is a consequence
of the fact that the set of isotropic vectors
of a semidefinite Hermitian symmetric form is a complex linear
space.\par
To complete the proof, we 
need to show that, if $\xi\in{H}^{\oplus}M$, we
have $\mathcal{L}_{\xi}(Z,\bar{Z})=0$ for all
$Z\in\mathfrak{Z}(M)\cap\overline{\mathbb{E}_{\mathfrak{Z}}(M)}$.
Having fixed $\xi\in{H}^{\oplus}M$, we can argue in a small coordinate
patch $U$ about its base point $p=\pi(\xi)$.
 By the assumption that $\mathfrak{Z}(M)$ is formally
integrable, we can choose real coordinates $x_1,\hdots,x_m$,
centered at $p$, such that, for a pair of nonnegative integers $h,\ell$ with
$2h+\ell\leq{m}$, setting
$z_j=x_j+i\, x_{h+j}$ for $j=1,\hdots,h$,
a system of generators of $\mathfrak{Z}(M)$ in $U$ is given by the
vector fields
\begin{equation}
  \label{eq:Bbd}
  \begin{cases}
    L_j=\frac{\partial}{\partial\bar{z}_j}+
L'_j&\text{for}\; j=1,\hdots,h,\\
L_{h+j}=\frac{\partial}{\partial{x_{m+1-j}}}&\text{for}\; j=1,\hdots,\ell.
  \end{cases}
\end{equation}
Here $\frac{\partial}{\partial\bar{z}_j}=\frac{1}{2}\left(
\frac{\partial}{\partial{x}_j}+i\frac{\partial}{\partial{x}_{h+j}}\right)$,
$L'_j\in\mathfrak{X}^{\mathbb{C}}(U)$ satisfies $L'_j(0)=0$,
and
\begin{equation}
  \label{eq:Bbe}
  [L_i,L_j]=0\quad\text{for}\quad 1\leq i,j\leq h+\ell.
\end{equation}
This is obtained by first noticing that the real vector fields in
$\mathfrak{Z}(M)$ are a formally integrable distribution of real vector
fields. By the classical Frobenius theorem, we can choose a system
of local coordinates in which this real distribution is locally
generated by the $L_{h+1},\hdots,L_{h+\ell}$ above. By linear algebra
we can obtain, in a neighborhood of $p$, from any basis of $\mathfrak{Z}(M)$ that includes $L_{h+1},\hdots,L_{h+\ell}$,
a new one in which the
$L_1,\hdots,L_h$ have the property that
$L_j-\frac{\partial}{\partial{x}_j}$
does not contain 
either $\frac{\partial}{\partial{x}_i}$ for $i=1,\hdots,h$, or
$\frac{\partial}{\partial{x}_{m+1-i}}$ for $i=1,\hdots,\ell$. By this choice
we obtain \eqref{eq:Bbe}. Let
\begin{equation}
  \label{eq:Bbg}
  L_j=\frac{\partial}{\partial\bar{z}_j}+\sum_{i=1}^{m-\ell}a_{j}^i
\frac{\partial}{\partial{x}_i},\quad\text{with}\; a_{j}^i(0)=0.
\end{equation}
We can assume, by a change of coordinates, that
\begin{equation}
  \label{eq:Bbh}
  L_j(a_r^i)(0)=\frac{\partial{a_r^i(0)}}{\partial\bar{z}_j}=
0\quad\text{for}\quad j,r=1,\hdots,h,\; i=2h+1,\hdots,m-\ell.
\end{equation}
In fact, by the formal integrability condition, it follows that
\begin{equation}
  \label{eq:Bbi}
  \frac{\partial}{\partial{\bar{z}}_j}+
\sum_{i=2h+1}^{m-\ell}\left(
\sum_{r=1}^h\frac{\partial{a^i_j(0)}}{\partial{\bar{z}}_r}\bar{z}_r\right)
\frac{\partial}{\partial{x}_i},\quad\text{for}\; j=1,\hdots,h,
\end{equation}
are  commuting vector fields, and, since they also commute with
their conjugates, by a change of the coordinates $x_1,\hdots,x_{m-\ell}$
we can obtain a new coordinate system for which 
\eqref{eq:Bbh} is also
satisfied.
Let 
$\xi\in\mathbb{R}^m$ be such that $\ell_j(0,\xi)=0$ for 
$j=1,\hdots,h+\ell$, where by $\ell_j(x,\xi)$ we indicate the symbols
of the differential operators $L_j$. This means that the components
$\xi_i$ of $\xi$ are zero for $1\leq{i}\leq{2h}$ and $(m-\ell)<i\leq{m}$.
By the formal integrability condition, there is a second degree
homogeneous polynomial
$q_{\xi}(x)\in\mathbb{C}[x_1,\hdots,x_{m-\ell}]$ such that
\begin{equation}
  \label{eq:Bbf}
  L_j(i\langle{x},\xi\rangle + q_{\xi}(x))=O(|x|^2)\quad \text{for}\; x\to{0}.
\end{equation}
Next we observe that, by identifying $\xi$ with the corresponding element
in $T^*_0\mathbb{R}^m$, and setting
$v_{\xi}=(i\langle{x},\xi\rangle + q_{\xi}(x))$, we obtain
\begin{equation*}\begin{aligned}
  0=d^2(i\langle{x},\xi\rangle + q_{\xi}(x))(Z,\bar{Z})=
Z\bar{Z}(v_{\xi})(0)-\bar{Z}Z(v_{\xi})(0)-\mathcal{L}_{\xi}(Z,\bar{Z}),
\\
\forall Z\in\mathfrak{X}^{\mathbb{C}}(U).
\end{aligned}
\end{equation*}
Thus, in particular, $Z\bar{Z}v_{\xi}(0)=\mathcal{L}_{\xi}(Z,\bar{Z})\geq{0}$
if $Z\in\mathfrak{Z}(M)$ and $\xi\in{H}^{\oplus}_0M$.
By \eqref{eq:Bbh}, we have $Z\bar{Z}v_{\xi}(0)=Z\bar{Z}q_{\xi}(0)$.
Consider the expression for $q_{\xi}$ as a polynomial in
$z_1,\hdots,z_h,\bar{z}_1,\hdots,\bar{z}_h,x_{2h+1},\hdots,x_{m-\ell}$,
\begin{equation}
  \label{eq:Bbj}\begin{aligned}
  q_{\xi}(x)=q^{2,0,0}_{\xi}(z,z)+q^{1,1,0}_{\xi}(z,\bar{z})+
q^{0,2,0}_{\xi}(\bar{z},\bar{z})+q^{1,0,1}_{\xi}({z},x'')\\
+
q^{0,1,1}_{\xi}(\bar{z},x'')+q^{0,0,2}_{\xi}(x'',x''),
\end{aligned}
\end{equation}
where $x''=(x_{2h+1},\hdots,x_{m-\ell})$. The assumption that
$\mathcal{L}_{\xi}\geq{0}$ means that $q^{(1,1,0)}_{\xi}(z,\bar{z})\geq{0}$.
We can add to $q_{\xi}$ any second degree homogeneous polynomial $f$
in $\mathbb{C}[z,x'']$, since $L_j(f)=O(|x|^2)$ for any such polynomial.
In this way, we obtain a new $q_{\xi}$, still satisfying \eqref{eq:Bbf},
with the property that
\begin{equation}
  \label{eq:Bbk}
  \mathrm{Re}(q_{\xi})(x)\geq{0},\;\forall x\in\mathbb{R}^m.
\end{equation}
Fix any real valued function $\chi\in\mathcal{C}^{\infty}_0(\mathbb{R}^m)$
with $\chi(x)=1$ for $|x|\leq{1}$, $0\leq\chi(x)\leq{1}$ in 
$\mathbb{R}^m$, and $\chi(x)=0$ for $|x|\geq{2}$.
For large $\tau>0$ the function
\begin{equation}
  \label{eq:Bbl}
  u_{\tau}=\root{6}\of{\tau^m}
\,\chi(x\root{3}\of{\tau})e^{-\tau(i\langle x,\xi\rangle+q_{\xi}(x))}
\end{equation}
belongs to $\mathcal{C}^{\infty}_0(U)$. We have
\begin{equation}\label{eq:Bbp}
  \|u_\tau\|_0^2=\int_{\mathbb{R}^m}e^{-2\root{3}\of{\tau}q_{\xi}(x)}
\chi^2(x)dx\leq\int_{\mathbb{R}^m}
\chi^2(x)dx,
\end{equation}
because of \eqref{eq:Bbk}. 
If $Z\in\mathfrak{Z}(M)$, we have:
\begin{equation*}
  \begin{aligned}
    |Z(u_\tau)|^2&=\root{3}\of{\tau^m}\left|-\tau 
\chi(x\root{3}\of{\tau}/2)Z(v_{\xi})+
\root{3}\of{\tau}[Z(\chi)](x\root{3}\of{\tau})\right|^2 e^{-2\tau
\mathrm{Re}\,{q}_{\xi}(x)}\\
&\leq C_0 \root{3}\of{\tau^m}\cdot\tau^2\,|x|^4 \chi(x\root{3}\of{\tau}/2)
\quad\text{for}\;\tau\gg{1},
  \end{aligned}
\end{equation*}
with a positive constant $C_0$. Indeed $\chi(x\root{3}\of{\tau}/2)=1$
when $Z(u_{\tau})\neq{0}$, and we used the fact that 
$Z(v_{\xi})=O(|x|^2)$ and that $Z(\chi)=O(|x|^2)$.
Computing $\|Z(u_\tau)\|_0^2$ by making the change of coordinates
$y=x\root{3}\of{\tau}$ shows that,
with a constant $C_1>0$ independent of $\tau$,
\begin{equation}\label{eq:Bbm}
  \|Z(u_\tau)\|_0^2\leq C_1\root{3}\of{\tau^2}\quad\text{for}\quad \tau\gg{1}.
\end{equation}
On the other hand, for $u\in\mathcal{C}^{\infty}_0(U)$, we have,
with a constant $C_2>0$ independent of $u$,
\begin{equation}\label{eq:Bbn}
      \|\bar{Z}(u)\|_0^2\geq \|Z(u)\|_0^2-\mathrm{Re}([Z,\bar{Z}](u)|u)_0
-C_2\|u\|_0^2.
 \end{equation}
For $u=u_{\tau}$ and by using, while
computing the integral, the change of variables $y=x\sqrt{\tau}$, we obtain
\begin{equation}\label{eq:Bbo} \begin{aligned}
-\mathrm{Re}([Z,\bar{Z}](u_{\tau})|u_{\tau})_0=
  \tau\int \mathcal{L}_{\xi}(Z,\bar{Z})& e^{-2
\mathrm{Re}\,{q}_{\xi}(y)}\chi^2(y/\root{6}\of{\tau})dy\\
&+{O}(\root{3}\of{\tau^2})\qquad\text{for}\;\tau\gg{1}.
\end{aligned}
\end{equation}
Since $\chi^2(y/\root{6}\of{\tau})$ is increasing with $\tau$, we get,
with a constant $C_3\geq{0}$,
\begin{equation}\label{eq:Bbx} \begin{aligned}
-\mathrm{Re}([Z,\bar{Z}](u_{\tau})|u_{\tau})_0\geq
  \tau\int \mathcal{L}_{\xi}(Z,\bar{Z}) e^{-2
\mathrm{Re}\,{q}_{\xi}(y)}\chi^2(y)dy-C_3|\tau|^{\frac{3}{2}}\quad\\
\quad\text{for}\;\tau\gg{1}.
\end{aligned}
\end{equation}
By \eqref{eq:Bbp} and \eqref{eq:Bbm}, 
$(\sum_{j=1}^{h+\ell}\|L_j(u_{\tau})\|_0^2+\|u\|_0^2)$
is ${O}(\root{3}\of{\tau^2})$ for $\tau\to\infty$.
Thus, 
by \eqref{eq:Bbn} and \eqref{eq:Bbx} we obtain that
$\mathcal{L}_{\xi}(Z,\bar{Z})=0$ for 
$Z\in\mathfrak{Z}(M)\cap\overline{\mathbb{E}_{\mathfrak{Z}}(M)}$.
\end{proof}
\begin{rmk}
By a slight variant of the proof of Proposition \ref{prop:Bbb} we obtain:
\textit{
  Assume that $\mathfrak{Z}(M)$ is formally integrable and that
$\mathfrak{Z}(M)$ and $\mathfrak{Z}(M)\cap\overline{\mathfrak{Z}(M)}$
are both distributions of constant rank. 
Let $p\in{M}$.
If there exists
$\xi\in{H}^0_pM$ such that $\mathcal{L}_{\xi}$ is definite on a
complement of $Z_pM\cap\overline{Z_pM}$ in
$Z_pM$, then $\mathfrak{Z}(M)$ is not subelliptic at $p$.}
\end{rmk}
\begin{prop}\label{pp:pcj}
For every distribution of complex vector fields $\mathfrak{Z}(M)$,
we have
\begin{equation}
  \label{eq:pce}
  \Theta_{\mathfrak{Z}}(M)\subset\mathbb{K}_{\mathfrak{Z}}(M). 
\end{equation}
Assume that 
\begin{align}
\label{eq:Bcb}
\mu_0(p)&=\mathrm{dim}_{\mathbb{R}}(Z_pM+\overline{Z_pM})\\
\label{eq:pcda} 
\delta_0({p})
&=
\mathrm{dim}_{\mathbb{R}}
\left\langle H^{\oplus}_pM\right\rangle,\\
\label{eq:pcdb}
\nu_0(p)&=\mathrm{dim}_{\mathbb{C}}\{Z(p)\in{Z}_pM
\mid Z\in\mathbb{K}_{\mathfrak{Z}}(M)\}
\end{align}
are constant in $M$.  Then
\begin{align}
  \label{eq:pch}
 \Theta_{\mathfrak{Z}}(M)&=\mathbb{K}_{\mathfrak{Z}}(M).\\
\intertext{Finally, if $\mathfrak{Z}(M)$ is formally integrable and of
constant rank,} 
\label{eq:Bca}
\mathbb{E}_{\mathfrak{Z}}(M)&=\mathfrak{Z}(M)+
\overline{\mathbb{K}_{\mathfrak{Z}}(M)},\\ \label{eq:Bcbb}
\mathbb{A}_{\mathfrak{Z}}(M)&=\{Z+\bar{Z}\mid Z\in\mathbb{K}_{\mathfrak{Z}}(M)\}.
\end{align}
\end{prop}
\begin{proof} By the definition of $\Theta_{\mathfrak{Z}}(M)$, we have
$\mathcal{L}_{\xi}(Z,\bar{Z})=0$ for all 
$\xi\in{H}^{\oplus}M$ and
$Z\in\Theta_{\mathfrak{Z}}(M)$. Thus \eqref{eq:pce} is always valid.\par
To prove that, under the additional assumptions,
we have the equality \eqref{eq:pch},
we apply
an argument similar to the one 
employed in \cite[Theorem 2.5]{HN00}.\par
By the constancy of $\mu_0(p)$, the characteristic set $H^0M$ is a smooth
real vector bundle on $M$. Then the assumption that $\delta_0(p)$ is constant
implies that $H^\oplus{M}$ generates a smooth linear subbundle of $H^0M$,
and therefore
the quotient $H^0M/\langle{H^{\oplus}M}\rangle$
is a smooth real linear bundle on $M$.
\par
Since $\nu_0(p)$ is constant,
\begin{equation*}
  M\ni{p}\to{K}_pM=\{Z(p)\in{Z}_pM
\mid Z\in\mathbb{K}_{\mathfrak{Z}}(M)\}
\end{equation*}
is a complex vector bundle $KM$ on $M$.
The map $\xi\to\mathcal{L}_{\xi}|_{KM}$ is injective from the quotient bundle
$H^0M/\langle{H^{\oplus}M}\rangle$ to the bundle 
$\mathrm{Herm}(KM)$
of Hermitian symmetric forms
on $KM$. We denote by $LM$ the image bundle. \par
 The dual bundle $\mathrm{Herm}^*(KM)$ 
of $\mathrm{Herm}(KM)$
is the real 
linear subbundle
of $KM\otimes_M\overline{KM}$ generated by the elements of the
form $Z(p)\otimes\bar{Z}(p)$, for $p\in{M}$ and $Z(p)\in{K}_pM$.
The annihilator bundle $L^0M$ of $LM$ in $\mathrm{Herm}^*(KM)$
contains, for all $p\in{M}$, positive definite elements of
$\mathrm{Herm}^*(KM)$. Since the positive definite elements of
$L^0M$ form an open set in $L^0M$, it is easy to construct a global
section $\mathfrak{z}$ of $L^0M$ with $\mathfrak{z}(p)>0$ for all $p\in{M}$,
by first constructing local sections, and then patching them together
by a partition of unity. If 
$p_0\in{M}$ and $Z\in\mathbb{K}_{\mathfrak{Z}}(M)$, 
with $Z(p_0)\neq{0}$, by the standard Gram-Schmidt orthogonalization process,
we can find a smooth function $\phi\in\mathcal{E}(M)$ with
$\phi(p_0)\neq{0}$ and sections 
$Z_2,\hdots,Z_k\in\mathbb{K}_{\mathfrak{Z}}(M)$,
such that, by setting $Z_1=\phi\cdot{Z}$, we have
 $\mathfrak{z}(p)=\sum_{j=1}^kZ_j(p)\otimes\bar{Z}_j(p)$ for
$p$ in an open neighborhood $U$
of $p_0$ in $M$, where $k$ is the rank of  the complex bundle $KM$. 
The equality
\begin{equation*}
  \sum_{j=1}^k\mathcal{L}_{\xi}(Z_j,\bar{Z}_j)=0\quad\forall 
p\in{U},\;\forall\xi\in{H}^0_pM
\end{equation*}
implies that
\begin{equation*}
  \sum_{j=1}^k[Z_j,\bar{Z}_j]\in\mathfrak{Z}(U)+\overline{\mathfrak{Z}(U)}.
\end{equation*}
By repeating this argument for a set of elements
of $\mathbb{K}_{\mathfrak{Z}}(M)$ 
whose values at $p_0$ give a
basis for $K_{p_0}M$, we prove that
every $Z\in\mathbb{K}_{\mathfrak{Z}}(M)$ coincides, on an open neighborhood
of $p_0$ in $M$, 
with the restriction of an element of $\Theta_{\mathfrak{Z}}(M)$.
Moreover, the number of
elements of $\varTheta_{\mathfrak{Z}}(M)$ that are needed
in \eqref{eq:nd} does not exceed $\nu=k^2-k$.
Thus, if $Z\in\mathbb{K}_{\mathfrak{Z}}(M)$,
we can find an open covering
$\{U_a\}$ of $M$ and, for each $a$, some 
$Z_2^{(a)},\hdots,Z_{\nu}^{(a)}\in\Theta_{\mathfrak{Z}}(M)$, for which
\begin{equation*}
  [Z,\bar{Z}]+\sum_{j=1}^{\nu}[Z^{(a)}_j,
\bar{Z}^{(a)}_j]\in\mathfrak{Z}(U_a)+\overline{\mathfrak{Z}(U_a)}
,\quad\forall a.
\end{equation*}
We can assume that the covering $\{U_{a}\}$ has a finite index. 
Then, by using a partition of unity $\{\chi_a\}$ subordinated to
$\{U_{a}\}$,
and summing together vector fields $\chi_a\cdot{Z}_j^{(a)}$ with
disjoint supports, we end up with a finite subset
$Z_1,\hdots,Z_r$ of vector fields in $\mathfrak{Z}(M)$
such that
\begin{equation*}
  \mathcal{L}_{\xi}(Z,\bar{Z})+\sum_{j=1}^r\mathcal{L}_{\xi}(Z_j,
\bar{Z}_j)=0,
\end{equation*}
showing that $Z\in\Theta_{\mathfrak{Z}}(M)$.\par
When moreover $\mathfrak{Z}(M)$ is formally integrable
and has constant rank, we obtain
\eqref{eq:Bca} as a consequence of Lemma \ref{lm:ai} and
Proposition \ref{prop:Bbb}.
\end{proof}
\begin{defn}\label{df:cra}
  We say that a \emph{higher Levi form concavity condition} is
  satisfied at the point $p\in{M}$ if
  \begin{equation} \left\{
    \begin{gathered}
      \label{eq:cra}
      \forall\xi\in{H}^0_pM\setminus\{0\}\;\text{with}\;\mathcal{L}_\xi\geq0,\\
\exists
      Z_0\in\mathfrak{Z}(M),Z_1,\hdots,Z_r\in
\mathbb{K}_{\mathfrak{Z}}(M)+\overline{\mathbb{K}_{\mathfrak{Z}}(M)}\\
      \text{with}\;
      i\xi([Z_1,\hdots,Z_r,\bar{Z}_0])\neq{0}.
    \end{gathered}\right.
  \end{equation}
\end{defn}
\begin{defn} \label{def:cg} We say that the distribution
$\mathfrak{Z}(M)$ is  \emph{regular} at
a point $p_0\in{M}$ if its rank, and
the functions
$\mu_0(p)$ of \eqref{eq:Bcb},
$\delta_0(p)$ of \eqref{eq:pcda},
$\nu_0(p)$ of \eqref{eq:pcdb}, are all constant in 
an open neighborhood $U$ of $p_0$.
\end{defn}
Since the rank of $\mathfrak{Z}(M)$, and the functions $\mu_0$ and
$\nu_0$ are all integral valued and semicontinuous, and
$\delta_0$ is semicontinuous on the dense open subset where $\mu_0$
is constant,
 the set of regular points of
$\mathfrak{Z}(M)$
is open and dense in $M$.
\begin{prop}\label{pp:crx}
Let $p_0$ be a regular point for $\mathfrak{Z}(M)$. Then 
\eqref{eq:cra} implies \eqref{eq:be}. If in addition we assume that
$\mathfrak{Z}(M)$ is formally integrable, then the two conditions
\eqref{eq:cra} and  \eqref{eq:be} are equivalent.
\end{prop}
\begin{proof} The regularity assumption implies that all $\xi\in{H}^0_{p_0}M$ with
$\ker\mathcal{L}_{\xi}\supset\Theta_{\mathfrak{Z}}(M)$ belong to the linear span
$\left\langle H^{\oplus}_{p_0}M\right\rangle$. Thus condition
\eqref{eq:pcb} holds if $\xi\in{H}^0_{p_0}M\setminus\{0\}$ and
$\ker\mathcal{L}_{\xi}\supset\Theta_{\mathfrak{Z}}(M)$. When
$\ker\mathcal{L}_{\xi}\not\supset\Theta_{\mathfrak{Z}}(M)$, the restriction
of $\mathcal{L}_{\xi}$ to $\Theta_{\mathfrak{Z}}(M)$ is semidefinite, and hence
there are $Z_1,Z_2\in\Theta_{\mathfrak{Z}}(M)$ with 
$\mathcal{L}_{\xi}(Z_1,\bar{Z}_2)
=i\xi([Z_1,\bar{Z}_2])\neq{0}$.

Then because of Propositions \ref{prp:ca} and \ref{pp:pcj}, 
and Lemma \ref{lm:ai}, we obtain that \eqref{eq:cra} implies 
condition \eqref{eq:be} at regular points of $\mathfrak{Z}(M)$.\par
If in addition $\mathfrak{Z}(M)$ is formally integrable, then by the equality
\eqref{eq:Bca}, the opposite implication is also true.
\end{proof}

\section{Pullbacks of Distributions}\label{sec:map}
Let $M,N$ be smooth real manifolds, and $N\xrightarrow{\varpi}M$
a smooth submersion. 
\par 
Given a distribution of complex vector fields $\mathfrak{Z}(M)$,
its pullback $[\varpi^*\mathfrak{Z}](N)$ 
consists of all $W\in\mathfrak{X}^{\mathbb{C}}(N)$ with the property
that
\begin{equation}
  \label{eq:Caa}
  \forall{U}^\mathrm{open}\subset{M},
\forall\sigma\in\mathcal{C}^{\infty}(U,N),\;\text{with}\;
\varpi\circ\sigma=\mathrm{id}_U,\;
d\varpi^{\mathbb{C}}(W\circ\sigma)\in\mathfrak{Z}(U).
\end{equation}
Let 
\begin{align}\label{eq:mpa}
V^{\varpi}N&=\{v\in{T}N\mid d\varpi(v)=0\}&&\text{be the vertical bundle, and}
\\\label{eq:mpha}
  \mathfrak{V}^{\varpi}(N)&=\mathcal{C}^{\infty}(N,V^{\varpi}N)&&\text{the vertical 
distribution}.
\end{align}

\begin{lem}
Let $N\xrightarrow{\varpi}M$ be a smooth submersion. Then
\begin{equation}
  \label{eq:mphb}
 \mathfrak{V}^{\varpi}(N)\subset[\varpi^*\mathfrak{Z}](M)
\quad\text{and}\quad  [\mathfrak{V}^{\varpi}(N),[\varpi^*\mathfrak{Z}](N)]
\subset[\varpi^*\mathfrak{Z}](N).
\end{equation}
\end{lem}
\begin{proof} We note that $[\varpi^*\mathfrak{Z}](N)$ is 
the space of global sections of a fine sheaf
of left $\mathcal{C}^{\infty}$-modules. Thus 
by localization we can reduce the discussion
to the case where $N=M\times\Omega$ 
for an open subset
$\Omega$ of a Euclidean space $\mathbb{R}^k$, and $\varpi$ is the
projection onto the first factor, in which situation
the statement is trivial.
\end{proof}

\begin{defn}
If $N\xrightarrow{\iota}M$ is a smooth immersion, we define the
\emph{pullback} $[\iota^*\mathfrak{Z}](N)$ of $\mathfrak{Z}(M)$ to $N$ to be 
the set of complex vector fields $Z'\in\mathfrak{X}^{\mathbb{C}}(N)$
having the following property
\begin{equation}
  \label{eq:mpe}\left\{\begin{gathered}
  \forall q_0\in{N},\;\exists {V}^{\mathrm{open}}\subset{N}\;\text{with}\;
V\ni{p_0},\;\text{and}\;
Z\in\mathfrak{Z}(M)\\
\text{s.t.}\quad d\iota(q)(Z'(q))=Z(\iota(q))\;\forall q\in{V}.
\end{gathered}\right.
\end{equation}
\end{defn}
Let $M,N$ be smooth manifolds.
A smooth map
$N\xrightarrow{\phi}M$ is a \emph{submersion onto its image} if
there exists a smooth manifold $S$, a submersion $N\xrightarrow{\varpi}S$
and an immersion $S\xrightarrow{\iota}M$ that factorize $\phi$, i.e.
that make the following diagram commute:
\begin{equation}
  \label{eq:mpd}
  \begin{CD}
    N @>\phi>> M\\
@V{\varpi}VV  @AA{\iota}A\\
S @= S
  \end{CD}
\end{equation}
\begin{defn}
The pullback of $\mathfrak{Z}(M)$ by a map $N\xrightarrow{\phi}M$,
which is a submersion onto its image, is the distribution of
complex vector fields
\begin{equation}
  \label{eq:mpf}
  [\phi^*\mathfrak{Z}](N)=[\varpi^*[\iota^*\mathfrak{Z}]](N),
\end{equation}
where $\varpi$ and $\iota$ are the maps in \eqref{eq:mpd}.
\end{defn}
\begin{rmk}
If $N$ is an open subset of $M$ and $\iota:N\hookrightarrow{M}$ is the
inclusion, then $\iota^*(\mathfrak{Z})(N)=\mathfrak{Z}(N)$ is the distribution
on $N$ that is generated by the restrictions to $N$ of the vector
fields $Z\in\mathfrak{Z}(M)$.\par
More generally, if $N\subset{M}$ is a smooth submanifold, 
and $\iota:N\hookrightarrow{M}$ 
the embedding map, then $[\iota^*\mathfrak{Z}](N)$
is the distribution generated by the restrictions to $N$ of the vector
fields $Z\in\mathfrak{Z}(M)$ with $Z(p)\in{T}_p^{\mathbb{C}}N$ 
for every~$p\in{N}$.
\end{rmk}
\begin{defn}\label{def:mpe}
Let $M,\, N$ be smooth manifolds, and $\mathfrak{Z}_M(M)$,
$\mathfrak{Z}_N(N)$ distributions of complex vector fields on $M$ and $N$,
respectively. Let $N\xrightarrow{\phi}{M}$ be a submersion onto its image.
We say that $\phi$ is a \emph{$Z$-morphism} if
\begin{equation}
  \label{eq:mpg}
  \mathfrak{Z}_N(N)\subset[\phi^*\mathfrak{Z}_M](M).
\end{equation}

\end{defn}
\begin{rmk}\label{rm:mpf}
We keep the notation of  Definition \ref{def:mpe}, and let $\varpi$,
$\iota$ be the maps in \eqref{eq:mpd}. Set 
$\mathfrak{Z}_S(S)=[\iota^*\mathfrak{Z}_M](S)$. Then 
$N\xrightarrow{\phi}M$ is a $Z$-morphism if and only if 
$N\xrightarrow{\varpi}S$ is a $Z$-morphism.
\end{rmk}
\begin{lem}\label{lem:Dg}
Let $M,N$ be two smooth real manifolds, and $N\xrightarrow{\varpi}M$
be a smooth submersion. Then 
\begin{align}
\label{eq:Daa}
\mathbb{E}_{\varpi^*{\mathfrak{Z}_{M}}}(N)&=[\varpi^*\mathbb{E}_{{\mathfrak{Z}_{M}}}](N)\\
\label{eq:Dab}
\mathbb{K}_{\varpi^*{\mathfrak{Z}_{M}}}(N)&=[\varpi^*\mathbb{K}_{{\mathfrak{Z}_{M}}}](N)\\
  \label{eq:mph}
  \Theta_{\varpi^*{\mathfrak{Z}_{M}}}(N)&=[\varpi^*\Theta_{{\mathfrak{Z}_{M}}}](N).
\end{align}
\end{lem}
\begin{proof}
Again the statement becomes trivial after, by localization, we reduce to
the case where $N=M\times\Omega$ with $\Omega$ open in $\mathbb{R}^k$ and
$\varpi$ being the projection onto the first factor.
\end{proof}

\begin{prop}
  \label{pp:da}
Let $M$ and $N$ be smooth real manifolds, with assigned complex valued
distributions of smooth 
complex vector fields $\mathfrak{Z}_M(M)$ and
$\mathfrak{Z}_N(N)$, respectively.
Let $N\xrightarrow{\phi}{M}$ be a smooth submersion onto the image and
a $Z$-morphism. Let $q_0\in{N}$ and $p_0=\phi(q_0)$ be
a regular point of $\mathfrak{Z}_M(M)$, according to Definition \ref{def:cg}.
Assume that
\begin{enumerate}
\item $\mathfrak{Z}_N(N)$ satisfies the higher Levi form concavity condition
 \eqref{eq:cra}  at the point $q_0$;
\item the pullback $\phi_{q_0}^*:H^0_{p_0}M\to{H}^0_{q_0}N$ 
is injective.
\end{enumerate}
Then 
conditions \eqref{eq:cra} and \eqref{eq:be} at $p_0$ are valid 
for $\mathfrak{Z}_M(M)$.
\end{prop}
\begin{proof}
Using Remark \ref{rm:mpf} we shall split the proof by separately considering
the case in which $\phi$ is a submersion and the case where $\phi$ is
an immersion.\par
First we assume that $N\xrightarrow{\phi}{M}$ is a smooth submersion.
 In this case 
$\mathbb{E}_{\mathfrak{Z}_N}(N)\subset
\mathbb{E}_{\phi^*\mathfrak{Z}_M}(N)$. Moreover,
(2) is automatically satisfied 
because $\mathfrak{Z}_N(N)\subset[\phi^*\mathfrak{Z}_M](N)$. 
The statement is trivially true, as it is easily checked by
reducing it to the case where $N=M\times\Omega$, with
$\Omega$ open in $\mathbb{R}^k$ and $\varpi$ the projection onto the
first coordinate.
\par
To complete the proof of the general case, 
it suffices, by localization about $q_0$, to consider the case
in which $N=S$ is a smooth submanifold of $M$, 
and $\mathfrak{Z}_N(N)$ are the restrictions to $N$ of 
elements of $\mathfrak{Z}_M(M)$ with real and imaginary parts
tangent to $N$ at all points of $N$.
By (2), \eqref{eq:cra} for $\mathfrak{Z}_N(N)$ at $q_0=p_0$, implies 
that \eqref{eq:cra} is satisfied at $p_0$ for
$\mathfrak{Z}_M(M)$. By the regularity assumption, this implies
\eqref{eq:be} at $p_0$ for $\mathfrak{Z}_M(M)$.
\end{proof}

\section{Hypoellipticity for some differential operators
of the first and of the second order}\label{sec:ope}
We keep the notation of \S\ref{sec:a}, \S\ref{sec:C}. 
Theorems \ref{tm:cca}, Corollary \ref{cor:Db}, and Theorem \ref{tm:ell}
below, which concern systems of first order partial differential operators, and
second order partial differential operators closely related to
sums of squares of vector fields,
directly follow from the assumption that
$\mathfrak{Z}(M)$ be subelliptic. 
The other results of this section, namely
Theorems \ref{tm:cpd} and \ref{tm:bh},
refer to generalized parabolic second order operators, and
are proved under conditions \eqref{eq:cpc} and \eqref{eq:essj}, 
respectively, that more directly involve the Lie structure of
$\mathfrak{Z}(M)$ with respect to some
\textit{generalized time vector field}.
\par
Let $E\xrightarrow{\pi}M$ be a complex vector bundle 
of rank $r$ on $M$, endowed with a $\mathbb{C}$-linear connection
\begin{equation}
  \label{eq:opcon}
  \nabla:\mathfrak{X}(M)\times\mathcal{C}^{\infty}(M,E)\to\mathcal{C}^{\infty}(M,E).
\end{equation}
In a local trivialization $E|_U\simeq U\times\mathbb{C}^r$,
the connection $\nabla$ is described by the datum of a 
$\mathfrak{gl}(r,\mathbb{C})$-valued smooth one form
$\gamma=(\gamma^{\alpha}_{\beta})
\in
\mathcal{C}^{\infty}(U,\mathfrak{gl}(r,\mathbb{C})\otimes
T^*M)$. Using upper Greek letters for the components of
the sections in $\mathcal{C}^{\infty}(U,E)\simeq\mathcal{C}^{\infty}(U,\mathbb{C}^r)$,
we have
\begin{equation}
  \label{eq:cca}
  \left(\nabla_X(\sigma)\right)^{\alpha}=X\sigma^{\alpha}+
\sum_{\beta=1}^r\gamma^{\alpha}_{\beta}
(X)\sigma^{\beta},\quad\text{for}\;\alpha=1,\hdots,r.
\end{equation}
By $\mathbb{C}$-linearity, we can define,
for each complex valued vector field
$Z\in\mathfrak{X}^{\mathbb{C}}$, a linear partial differential
operator
\begin{equation}
  \label{eq:ccc}
  \nabla_Z:\mathcal{C}^{\infty}(M,E)\ni\sigma\to \nabla_{\mathrm{Re}\,Z}(\sigma)+
i\nabla_{\mathrm{Im}\,Z}(\sigma)\in\mathcal{C}^{\infty}(M,E).
\end{equation}
Let us fix a smooth
Riemannian metric $g$ on $M$ and a smooth Hermitian metric $h$ on the fibers
of $E\xrightarrow{\pi}M$. Then we can define the formal adjoint
$\nabla_Z^*:\mathcal{C}^{\infty}(M,E)\to\mathcal{C}^{\infty}(M,E)$ by
\begin{equation}
  \label{eq:ccd}\begin{gathered}
  \int_M{h(\nabla^*_Zu,v)d\lambda_g}=\int_M{h(u,\nabla_Zv)d\lambda_g}\\
\forall u,v\in\mathcal{C}^{\infty}(M,E),\;\text{with}\;
\mathrm{supp}(u)\cap\mathrm{supp}(v)\Subset{M},
\end{gathered}
\end{equation}
where $d\lambda_g$ is the Lebesgue density on $M$ with respect to
the Riemannian metric~$g$.\par
We obtain an analogue of \cite[Theorem 4.1]{HN00}:
\begin{thm}\label{tm:cca}
Let $E\xrightarrow{\pi}M$ be a smooth complex vector bundle of rank $r$
on $M$ and 
$\nabla$ 
a $\mathbb{C}$-linear connection 
on $E\xrightarrow{\pi}M$. If $\mathfrak{Z}(M)$ is subelliptic
at a point $p_0\in{M}$, 
then any weak solution $u\in{L}^2_{\mathrm{loc}}(M,E)$
of 
\begin{equation}
  \label{eq:cce}
  \nabla_Zu\in\mathcal{C}^{\infty}(M,E)\qquad\forall Z\in\mathfrak{Z}(M)
\end{equation}
is equal, a.e. in an open neighborhood $U$ of $p_0$ in $M$, to
a smooth section of~$E\xrightarrow{\pi}M$. \par 
Formula \eqref{eq:cce} means that
for every $Z\in\mathfrak{Z}(M)$ there is a smooth section 
$f_Z\in\mathcal{C}^{\infty}(M,E)$ such that
\begin{equation}
  \label{eq:ccf}\begin{aligned}
  \int_M{h(u,\nabla_Z^*v)d\lambda_g}=\int_M{h(f_Z,v)d\lambda_g},
\qquad\\\quad
\forall v\in\mathcal{C}^{\infty}(M,E),\;\text{with}\;
\mathrm{supp}(v)\Subset{M}.
\end{aligned}
\end{equation}
\end{thm}
\begin{proof}
The statement is local. Therefore
by substituting for $M$ a
relatively compact open neighborhood of $p_0$,
we can assume that $\mathfrak{Z}(M)$ is generated by a finite set of
vector fields $L_1,\hdots,L_n$, and that,
for some $\epsilon>0$ and $C>0$, we have the estimate
\begin{equation}
  \label{eq:ccg}
\|v\|_{\epsilon}^2\leq
C\left(  \sum_{j=1}^n\|L_j(v)\|_0^2+\|v\|_0^2\right)
,\quad\forall v\in\mathcal{C}^{\infty}_0(M).
\end{equation}
We can also assume that $E$ is the trivial bundle 
$M\times\mathbb{C}^r$ on $M$, so that
\eqref{eq:cce} is equivalent to the system
\begin{equation}
  \label{eq:cch}
  L_ju^{\alpha}+\sum_{\beta=1}^ra^{\alpha}_{j\, \beta}(p)u^{\beta}=f^{\alpha}_j
\in\mathcal{C}^{\infty}(M),\quad\text{for}\;j=1,\hdots,n,\;\alpha=1,\hdots,r
,\end{equation}
with $a^{\alpha}_{j\,{{\beta}}}\in\mathcal{C}^{\infty}(M)$.
Defining
\begin{equation}
  \label{eq:cci}\left\{\begin{gathered}
  \mathfrak{d}:\mathcal{C}^{\infty}(M,\mathbb{C}^r)\to
\left[\mathcal{C}^{\infty}(M,\mathbb{C}^r)\right]^n\quad\text{by}\qquad
\qquad\\
\mathfrak{d}((u^{\alpha})_{\alpha=1,\hdots,r})=
\left(L_ju^{\alpha}+\sum_{\beta=1}^ra^{\alpha}_{j\,\beta}(p)u^{\beta}\right)_{
  \begin{smallmatrix}
    1\leq\alpha\leq{r}\\
1\leq{j}\leq{n}
  \end{smallmatrix}}
\end{gathered}\right.
\end{equation}
we obtain from \eqref{eq:ccg} that, with some new constant $C>0$
and the same $\epsilon>0$:
\begin{equation}
  \label{eq:ccj} \|v\|_{\epsilon}^2\leq
 C\left( \|\mathfrak{d}(v)\|^2_0+\|v\|^2_0\right),
\quad\forall v\in\mathcal{C}^{\infty}_0(M,\mathbb{C}^r).
\end{equation}
We have, for $\chi\in\mathcal{C}^{\infty}_0(M)$,
\begin{equation}
  \label{eq:cck}
  \left(\mathfrak{d}(\chi\,u)\right)^{\alpha}_j=
\chi\left(\mathfrak{d}(u)\right)^{\alpha}_j+(L_j(\chi)\,u^{\alpha}).
\end{equation}
In particular, if $u\in{L}^2_{\mathrm{loc}}(M,\mathbb{C}^r)$ is a
weak solution of \eqref{eq:cch}, we have that 
$\mathfrak{d}(\chi\,{u})\in \left[{L}^2(M,\mathbb{C}^r)\right]^n$.
By applying Friedrichs' theorem on the identity of the weak and
strong extensions of a first order partial differential operator,
we can find a sequence 
$\{v_\nu\}\subset\mathcal{C}^{\infty}_0(M,\mathbb{C}^r)$ such that
\begin{equation}
  \label{eq:ccl}
  v_\nu\rightarrow \chi\,u\;\text{in}\; L^2(M,\mathbb{C}^r)\;
\text{and}\; \mathfrak{d}(v_\nu)\rightarrow\mathfrak{d}(\chi\,u)\;
\text{in}\; \left[L^2(M,\mathbb{C}^r)\right]^n. 
\end{equation}
The subelliptic estimate \eqref{eq:ccj} yields a uniform bound for the
Sobolev $\epsilon$-norm $\|v_{\nu}\|_{\epsilon}$. This implies that
$\chi\,u\in{W}^{\epsilon}(M,\mathbb{C}^r)$, where
${W}^{\epsilon}(M,\mathbb{C}^r)$ denotes the Sobolev space of 
$L^2$-vector valued functions that have $L^2$-derivatives of the
positive real order $\epsilon$. Hence 
$u\in{W}^{\epsilon}_{\mathrm{loc}}(M,\mathbb{C}^r)$. \par
To show that $u\in\mathcal{C}^{\infty}(M,\mathbb{C}^r)$,
we use the Sobolev embedding theorem: it suffices to show that
$u\in{W}^{s}_{\mathrm{loc}}(M,\mathbb{C}^r)$ for all $s>0$.
Assume that we already know that this is true for some $s_0>0$.
If $\chi\in\mathcal{C}^{\infty}_0(M)$, then
\begin{equation*}
  \mathfrak{d}(\chi\,u)=[\mathfrak{d},\chi](u)+\chi\,\mathfrak{d}(u)
\in{W}^{s_0}(M,\mathbb{C}^r)
\end{equation*}
and has compact support. 
Fix any scalar pseudodifferential operator 
$\Lambda_q\in\Psi^q(M)$,  with $s_0-\epsilon<q\leq{s}_0$. 
Then, by the continuity
properties of classical pseudodifferential operators, we obtain
\begin{equation}
  \label{eq:ccm}
  \mathfrak{d}(\Lambda_q(\chi u))=[\mathfrak{d},\Lambda_q](\chi u)+
\Lambda_q(\mathfrak{d}(\chi u))
\in{L}^2_{\mathrm{loc}}(M,\mathbb{C}^r),
\end{equation}
because $\Lambda_q$ and 
the commutator $[\mathfrak{d},\Lambda_q](u)$ have orders $\leq{s}_0$
and 
$\chi\,u,\;\mathfrak{d}(\chi\,u)\in{W}^{s_0}(M)$ and have 
compact support in $M$.
Thus, by the argument above, 
$\Lambda_q(\chi\,u)\in{W}^{\epsilon}(M,\mathbb{C}^r)$. This implies that 
$\chi\,u\in {W}^{q+\epsilon}(M,\mathbb{C}^r)$, with ${q+\epsilon>s}$.
Since $\chi$ was an arbitrary smooth function with compact support
in $M$, this yields $u\in{W}^{q+\epsilon}_{\mathrm{loc}}(M,\mathbb{C}^r)$.
By recurrence we obtain that 
$u\in{W}^{s}_{\mathrm{loc}}(M,\mathbb{C}^r)$ for all $s>0$,
and hence is equal a.e. to a smooth section. The proof is complete.
\end{proof}
\begin{cor}\label{cor:Db}
Assume that $M$ is 
compact and that $\mathfrak{Z}(M)$ is subelliptic at all 
points $p\in{M}$. Then
\begin{enumerate}
\item the space
  \begin{equation}
    \label{eq:coa}
    \mathcal{O}_{\mathfrak{Z}}(M)=
\{\sigma\in\mathcal{C}^{\infty}(M,E)\mid \nabla_Z(\sigma)=0,\;
\forall Z\in\mathfrak{Z}(M)\}
  \end{equation}
is finite dimensional.
\item The map
  \begin{equation}
    \label{eq:cob}
    \mathcal{C}^{\infty}(M,E)\times\mathfrak{Z}(M)\ni (\sigma,Z)\to
(\nabla_Z(\sigma),Z)\in\mathcal{C}^{\infty}(M,E)\times\mathfrak{Z}(M)
  \end{equation}
has a closed range. By this we
mean that, if $\{\sigma_\nu\}$ is a sequence of sections
in $\mathcal{C}^{\infty}(M,E)$ and for each $Z\in\mathfrak{Z}(M)$ there is
$f_Z\in\mathcal{C}^{\infty}(M,E)$ such that $\nabla_Z(\sigma_\nu)$ converges
uniformly to $f_Z$ in $M$, then there exists 
a section $\sigma\in\mathcal{C}^{\infty}(M,E)$
with $\nabla_Z(\sigma)=f_Z$ for all $Z\in\mathfrak{Z}(M)$.
\end{enumerate}
\end{cor}
\begin{proof}
$(1)$ follows from the 
Sobolev embedding theorem, because
$\mathcal{O}_{\mathfrak{Z}}(M)$ is an $L^2$-closed subspace of 
$\mathcal{C}^{\infty}(M,E)$ on which
the $L^2$ and the $\epsilon$-Sobolev norm, for some $\epsilon>0$,
are equivalent. Finally,
$(2)$ is a consequence of the fact that,
since $M$ is compact,  
for a finite set $L_1,\hdots,L_n\in\mathfrak{Z}(M)$,
some $\epsilon>0$ and some $\mathrm{const}>0$,
on the $L^2$-orthogonal complement of $\mathcal{O}_{\mathfrak{Z}}(M)$,
we have the coercive
estimate
\begin{equation}
  \label{eq:coc}
  \|u\|_{\epsilon}^2\leq\mathrm{const}\left(\sum_{j=1}^n\|L_j(u)\|_0^2\right)
,\quad\forall u\in\mathcal{C}^{\infty}(M,E)\cap\left[\mathcal{O}_{\mathfrak{Z}}(M)\right]^{\perp}.
\qedhere\end{equation}
\end{proof} 

\begin{thm}\label{tm:ell} Let $\mathfrak{Z}(M)$ be a distribution of complex vector
fields.
\begin{enumerate}
\item
We can find a locally finite family $\{L_j\}\subset\mathfrak{Z}(M)$,
such that, for any choice
of $a\in\mathcal{C}^{\infty}(M)$, the second order operator
\begin{equation}
  \label{eq:opab}
  P(u)=a\,u+\sum_j{\bar{L}_jL_j}(u)
\end{equation}
is hypoelliptic at all points of
$M$ at which $\mathfrak{Z}(M)$ is subelliptic.
\item
If $\mathfrak{Z}(M)$ is finitely generated then, 
for any set of generators
$L_1,\hdots,L_n$
of $\mathfrak{Z}(M)$, 
and for any choice of $Z_0\in\mathfrak{Z}(M)+\overline{\mathfrak{Z}(M)}$ and 
$a\in\mathcal{C}^{\infty}(M)$, the operator 
\begin{equation}\label{eq:toa}
P(u)=\sum_{j=1}^n\bar{L}_jL_j(u)+Z_0(u)+a\,u
\end{equation}
is hypoelliptic at all points $p$ of $M$ where $\mathfrak{Z}(M)$ is subelliptic.
\end{enumerate}
\par
Let $M'$ be the open subset of $M$ 
of points $p$ where $\mathfrak{Z}(M)$ is subelliptic.
Then the operators $P$ of $(1)$ and $(2)$ satisfy the following:
\begin{equation}
  \label{eq:opaa}\left\{
  \begin{gathered}
    \forall U^{\mathrm{open}}\Subset{M}',\quad\exists \epsilon>0,\;
C>0\;\text{such that}\\
\|u\|_{\epsilon}^2\leq C\left(\left|(P(u)|u)_0\right|+\|u\|_0^2\right),\quad
\forall u\in\mathcal{C}^{\infty}_0(U).
  \end{gathered}\right.
\end{equation}
\end{thm}
\begin{proof}
Let $\{U_\nu\}$ be an open covering of $M'$ by relatively compact open
subsets, and, for each of them, let
$L_1^{(\nu)},\hdots,
L_{n_{\nu}}^{(\nu)}\in\mathfrak{Z}(M)$ be chosen in such a way that,
for suitable $\epsilon_\nu>0$,
$C_\nu>0$, we have the estimate
\begin{equation}
  \label{eq:opda}
  \|u\|^2_{\epsilon_\nu}\leq C_{\nu}
\left(\sum_{h=1}^{n_{\nu}}\|L_h^{(\nu)}(u)\|^2_{0}
+\|u\|_0^2\right),\quad\forall u\in\mathcal{C}^{\infty}_0(U_{\nu}).
\end{equation}
Take 
smooth functions $\chi_\nu\in\mathcal{C}^{\infty}_0(U_\nu)$
such that $\mathrm{supp}(\chi_{\nu})\Subset{U}_{\nu}$, the family
$\{\mathrm{supp}(\chi_{\nu})\}$ is locally finite, and
$\sum_{\nu}|\chi_{\nu}(p)|^2>0$ for all $p\in{M}'$. Then
$(1)$ holds with $\{L_j\}=\{\chi_{\nu}L^{(\nu)}_h\}$.\par
By \cite{KN65}, the
hypoellipticity of \eqref{eq:opab} and of \eqref{eq:toa}
is a consequence of \eqref{eq:opda}. Since it suffices to prove that
for each $p\in{M}'$, there is a small open neighborhood $U\Subset{M}'$
of $p$ for which 
\eqref{eq:opda} holds true,
we can reduce the proof 
to the case where $\mathfrak{Z}(M)$ is the
$\mathcal{C}^{\infty}(M)$-module generated by 
$L_1,\hdots,L_n\in\mathfrak{Z}(M)$, and
the operator $P$ is of the form \eqref{eq:toa}.
\par
By integration by parts we obtain
\begin{equation}\label{eq:opdff}
  \begin{aligned}
    -(Pu|u)_0=\sum_{j=1}^n\|L_j(u)\|_0^2-(L_0(u)|u)_0+(u|
L_{n+1}(u))_0
+(a'u|u)_0\\[-3pt]
\qquad\qquad\qquad\qquad\qquad\qquad
\forall u\in\mathcal{C}^{\infty}_0(M).
  \end{aligned}
\end{equation}
Since $L_1,\hdots,L_n$ generate $\mathfrak{Z}(M)$, 
 we obtain that \eqref{eq:ob}
is valid 
for every
relatively compact open subset $U\Subset{M}'$,
and this in turn,
together with \eqref{eq:opdff},
 implies \eqref{eq:opda}. 
The proof is complete.
\end{proof}
\begin{thm}\label{tm:cpd}
  We keep the notation of Definition \ref{def:hcg}. Assume that
$\mathfrak{Z}(M)$ is generated by a finite set $L_1,\hdots,L_n$ of
complex vector fields. Let $X_0\in\mathfrak{X}(M)$ be a real vector
field. Let 
\begin{align}
  \label{eq:cpa}
  \mathbb{T}'_0(M)&=\mathcal{C}^{\infty}(M,\mathbb{R})\cdot{X}_0+
\mathfrak{T}_{\mathfrak{Z}}(M)\\
\label{eq:cpd}
\mathbb{T}'_h(M)&=[\mathfrak{A}_{\mathfrak{Z}}(M),\mathbb{T}'_{h-1}(M)],\quad
\text{for}\; h\geq{1},\\\label{eq:cpe}
\mathfrak{T}'(M)&=\sum_{h=0}^{\infty}\mathbb{T}'_{h}(M).
\end{align}
In particular,
$\mathfrak{T}'(M)$ is the $\mathfrak{A}_{\mathfrak{Z}}(M)$-Lie-submodule of
$\mathfrak{X}(M)$ 
generated
by $X_0$ and $\mathfrak{T}_{\mathfrak{Z}}(M)$. Then, for any choice of
$Y_0\in\mathbb{T}^{(0)}_{\mathfrak{Z}}(M)$ and $a\in\mathcal{C}^{\infty}(M)$, 
the second order partial differential operator
\begin{equation}
  \label{eq:cpb}
  P(u)=a\cdot{u}+X_0(u)+iY_0(u)+\sum_{j=1}^n\bar{L}_j{L}_j(u)
\end{equation}
is hypoelliptic at all points $p\in{M}$ where
\begin{equation}
  \label{eq:cpc}
  \{X(p)\mid X\in\mathfrak{T}'(M)\}=T_pM.
\end{equation}
\end{thm}
We divide the proof of Theorem \ref{tm:cpd} into several steps. First we prove
\begin{lem}
  Let $U\Subset{M}$ be an open set
and assume that there are
$\epsilon>0$, $C>0$  and $A_1,\hdots,A_r\in\Psi^0(U)$ such that
\begin{equation}\left\{\begin{aligned}
  \label{eq:cpdx}
  \|u\|_{\epsilon}^2+\sum_{j=1}^n\|L_j(u)\|^2_0\leq
C\left(\sum_{h=1}^r\left|(P(u)|A_h(u))_0\right|+\|u\|_0^2\right),\\
\quad
\forall u\in\mathcal{C}^{\infty}_0(U).
\end{aligned}\right.
\end{equation}
Then for every real $s\geq{0}$, 
and every open subset $U'$ with $U'\Subset{U}$,
there is a constant $C'=C(s,U')$ such that
\begin{equation}
  \label{eq:pcee}
  \|u\|_{s+\epsilon}^2+\sum_{j=1}^n\|L_j(u)\|_s^2\leq
C'\left(\|P(u)\|_s^2+\|u\|_0^2\right),\quad\forall u\in\mathcal{C}^{\infty}(U').
\end{equation}
\end{lem}
\begin{proof}
  Let $\Lambda_s\in\Psi^s(U)$ be elliptic. Then we have,
with real constants ${C_1>0}$, $C_2\geq{0}$, uniformly for
$u\in\mathcal{C}^{\infty}_0(U')$,
  \begin{equation*}
    \begin{aligned}
      \|u\|_{s+\epsilon}^2+\sum_{j=1}^n\|L_j(u)\|_s^2&\leq C_1\left(
\|\Lambda_s(u)\|_{\epsilon}^2+\sum_{j=1}^n\|\Lambda_s(L_j(u))\|_0^2\right)\\
&\leq C_1\left(\|\Lambda_s(u)\|_{\epsilon}^2
+\sum_{j=1}^n\|L_j(\Lambda_s(u))\|_0^2\right)+C_2\|u\|^2_s
    \end{aligned}
  \end{equation*}
Thus, using the inequality
\begin{equation*}\left\{\begin{gathered}
\forall \delta>0,\quad\exists C_\delta>0\quad\text{s.t}\\
  \|u\|_s^2\leq\delta\|u\|^2_{s+\epsilon}+C_{\delta}\|u\|_0^2,\quad
\forall u\in\mathcal{C}^{\infty}_0(U),
\end{gathered}\right.
\end{equation*}
we obtain that
the left hand side of \eqref{eq:pcee} is bounded by a constant times
\begin{equation*}
  \sum_{h=1}^r\left|(P(\Lambda_s(u))|A_h(\Lambda_s(u)))_0\right|+
\|u\|_0^2.
\end{equation*}
The operator $Q=P+\sum_{j=1}^n{L}_j^*L_j$ is an operator of the first
order, with principal part $X_0+iY_0$. Therefore if $A\in\Psi^0(U)$,
we obtain
\begin{equation*}
  \begin{aligned}
    (P(\Lambda_s(u))|A(\Lambda_s(u))_0&=
(\Lambda_s(Q(u))|A(\Lambda_s(u)))_0+{O}(\|u\|_s^2)\\
&\qquad -\sum_{j=1}^n(L_j(\Lambda_s(u))|L_j(\Lambda_s(A(u))))_0\\
&=(\Lambda_s(Q(u))|A(\Lambda_s(u)))_0+{O}(\|u\|_s^2)\\
&\quad + \sum_{j=1}^n(\bar{L}_j\Lambda_sL_j(u)|A(\Lambda_s(u)))_0+
{O}(\|L_j(u)\|_s\|u\|_s)\\
&\quad + \sum_{j=1}^n([\Lambda_s,L_j](u)|[L_j,A\circ\Lambda_s](u))_0\\
&\quad
+\sum_{j=1}^n([\Lambda_s,L_j](u)|A\circ\Lambda_s(L_j(u)))_0\\
&=(\Lambda_s(P(u))|A\circ\Lambda_s(u))_0+{O}(\|u\|_s^2+
\sum_{j=1}^n\|L_j(u)\|_s^2),\\
&\qquad\qquad\forall u\in\mathcal{C}^{\infty}(U'),
  \end{aligned}
\end{equation*}
where we use ${O}(N(u))$ to indicate some quantity whose
modulus is bounded by a constant times $N(u)$. This computation
yields
\begin{equation}
  \label{eq:pcef}
 \left\{
   \begin{gathered}
     \forall U'\Subset{U},\;\forall s\in\mathbb{R}_+,\; \exists
A^{(2s)}_1,\hdots,A^{(2s)}_r\in\Psi^{(2s)}(U),\;\exists C_s'>0\;\text{s.t}\\
\begin{aligned}
\|u\|_{s+\epsilon}^2+\sum_{j=1}^{n}\|L_j(u)\|_s^2\leq C_s'
\left(\sum_{h=1}^r\left|(P(u)|A^{(2s)}_h(u))_0\right|+\|u\|_0^2\right),\\\quad
\forall u\in\mathcal{C}^{\infty}_0(U').
\end{aligned}
\end{gathered}\right.
\end{equation}
Clearly \eqref{eq:pcef} implies \eqref{eq:pcee}.
\end{proof}
It is known (see e.g. \cite{Hor67, Ko05}) that
\begin{lem}\label{lm:cpf}
If \eqref{eq:pcee} is valid for all $s\in\mathbb{R}_+$ and all
open subset $U'\Subset{U}$, then $P$ is $\mathcal{C}^{\infty}$-hypoelliptic
in $U$.
\end{lem}
\begin{proof}[End of the proof of Theorem \ref{tm:cpd}]
By the previous Lemmas,
we only need to prove
\eqref{eq:cpdx}. First we note that, for all $u\in\mathcal{C}^{\infty}_0(U)$,
\begin{equation*}
  \begin{aligned}
    \left|\mathrm{Re}\,(X_0(u)|u)_0\right|&=\frac{1}{2}\left|
\int_UX_0(u\,\bar{u})d\lambda_g\right|\\
&=\frac{1}{2}
\left|\int_U{|u|^2X^*_0(1)d\lambda_g}\right|\leq
\frac{1}{2}\left(\sup_{p\in{U}}|X^*_0(1)|
\right)\,\|u\|_0^2
  \end{aligned}
\end{equation*}
Since, for some positive constant $C_0$ depending on $U$,
\begin{equation*}
\left|
(iY_0(u)+a\,u|u)_0\right|\leq C_0\|u\|_0\left(\|u\|_0+\sum_{j=1}^n\|L_j(u)\|_0
\right),\quad\forall u\in\mathcal{C}^{\infty}_0(U), 
\end{equation*}
we obtain, upon integrating by parts, that with a constant $C_1>0$,
\begin{equation}
  \label{eq:cpg}
  \sum_{j=1}^n\|L_j(u)\|_0^2\leq -2(P(u)|u)_0+C_1\|u\|_0^2,\quad
\forall u\in\mathcal{C}^{\infty}_0(U).
\end{equation}
Next we note that, taking for
$A_0\in\Psi^0(U)$ the composition with $X_0$ of an elliptic
pseudodifferential operator $\Lambda_{-1}\in\Psi^{-1}(U)$, we have,
with some constant $C_2>0$,
\begin{equation*}
  \begin{aligned}
    \|X_0(u)\|_{-\frac{1}{2}}^2&\leq C_2\left(|(X_0(u)|A_0(u))_0|+\|u\|_0^2
\right),\quad\forall u\in\mathcal{C}^{\infty}_0(U).
  \end{aligned}
\end{equation*}
We obtain, with constants $C_3,C_4>0$,
\begin{equation*}
  \begin{aligned}
    \left|(X_0(u)|A_0(u))\right|&
\leq\left|(P(u)|A_0(u))_0\right|+
\left|\sum_{j=1}^n(L_j(u)|L_j(A_0(u)))_0\right|\\&\qquad\qquad +C_3
\|u\|\left(\|u\|+\sum_{j=1}^n\|L_j(u)\|_0\right)\\
&\leq \left|(P(u)|A_0(u))_0\right|+
\sum_{j=1}^n\|L_j(u)\|^2_0
+C_4\|u\|_0^2,\\
&\qquad\qquad\qquad\forall{u}\in\mathcal{C}^{\infty}_0(U).
  \end{aligned}
\end{equation*}
To complete the proof, it suffices to show that, if $Y\in\mathfrak{X}(M)$
satisfies, for some $\delta>0$, for some $A_1,\hdots,A_r\in\Psi^0(U)$, 
and a constant $C'>0$, the estimate
\begin{equation}
  \label{eq:cph}
  \|Y(u)\|^2_{\delta-1}\leq{C'}\left(\sum_{h=1}^r\left|
(P(u)|A_h(u))_0\right|+\|u\|_0^2\right),\quad
\forall u\in\mathcal{C}^{\infty}_0(U),
\end{equation}
and $X\in\mathbb{A}_{\mathfrak{Z}}(M)$, then we have, with 
some $A'_1,\hdots,A'_{r'}\in\Psi^0(U)$, and constants 
$\delta'>0$, $C''>0$,
\begin{equation}
  \label{eq:cpi}\left\{\begin{aligned}
  \|[X,Y](u)\|^2_{\delta'-1}\leq{C''}\left(\sum_{h=1}^{r'}\left|
(P(u)|A_h'(u))_0\right|+\|u\|_0^2\right),\\
\quad\forall u\in\mathcal{C}^{\infty}_0(U).
\end{aligned}\right.
\end{equation}
Recall that the estimate \eqref{eq:Aac} holds for $Z=X$, and hence
\begin{equation}
  \label{eq:cpj}
  \|X(u)\|_0^2\leq{c}_X\left(\left|(P(u)|u)_0\right|+\|u\|_0^2\right)\,\quad
\forall u\in\mathcal{C}^{\infty}_0(U),
\end{equation}
with a constant $c_X>0$.
Then, with positive constants $c_i>0$ and a ${T\in\Psi^{\delta-1}(U)}$,
we get
\begin{equation*}
  \begin{aligned}
    \|[X,Y](u)\|_{\frac{\delta}{2}-1}^2&\leq
c_0\left(\left|([X,Y](u)|T(u))_0\right|
  +\|u\|_0^2\right), \quad \forall u\in\mathcal{C}^{\infty}_0(U).
  \end{aligned}
\end{equation*}
Assuming, as we can, that $\delta\leq\frac{1}{2}$, we obtain
\begin{equation*}
  \begin{aligned}
    \left|(XY(u)|T(u))_0\right|&\leq
\left|(T^*(Y(u))|X(u))_0\right|+
\left|([T,X]^*(Y(u))|u)_0\right|+c_1\|u\|_0^2\\
&\begin{aligned}
\leq c_2\|Y(u)\|_{\delta-1}\left(\|X(u)\|_0+\|u\|_0\right)+c_1\|u\|_0^2,\quad
\forall u\in\mathcal{C}^{\infty}(U).
\end{aligned}
\end{aligned}
\end{equation*}
The last term, in view of \eqref{eq:cpj},
can be estimated by the right hand side of \eqref{eq:cpi}.
Likewise
\begin{equation*}
  \begin{aligned}
    \left|(YX(u)|T(u))_0\right|&\leq
\left|(X(u)|T(Y(u)))_0\right|
+
\left|(X(u)|[T,Y](u))_0\right|
+ c_3\|u\|_0^2\\
&
\begin{aligned}
\leq c_4\|X(u)\|_0\left(\|Y(u)\|_{\delta-1}+\|u\|_0\right)+c_3\|u\|_0^2,\quad
\\
\forall u\in\mathcal{C}^{\infty}_0(U),
\end{aligned}
  \end{aligned}
\end{equation*}
and in view of \eqref{eq:cpj}, also this last term can be estimated
by the right hand side of \eqref{eq:cpi}. The proof is complete.
\end{proof}
\begin{rmk}
Theorem \ref{tm:cpd} is weaker than the analogous statement 
in \cite{Hor67} in the
case where  the $L_j$'s are real. Indeed, for $L_j\in\mathfrak{X}(M)$,
setting $X_0=L_0$,
our assumption requires that all commutators
$[L_{j_1},\hdots,L_{j_{r-1}},L_{j_r}]$ with $r\in\mathbb{Z}_+$, and 
$0\leq j_h\leq n$, and $j_h>0$ for $h<r$, span the tangent space $T_pM$.
The statement in \cite{Hor67}, also proved in \cite{K71} and \cite{ff81},
allows $j_h=0$ also for $1\leq{h}<r$. 
This motivates us to consider separately the special case where
$\mathbb{E}_{\mathfrak{Z}}(M)=\mathfrak{Z}(M)+\overline{\mathfrak{Z}(M)}$:
Theorem \ref{tm:bh} below generalizes the case where all the $L_j$'s are
real.
\end{rmk}
\begin{thm}\label{tm:bh} We keep the notation of Definition \ref{def:hcg}.
Assume that 
\begin{gather}
  \label{eq:essc}
L_1,\hdots,L_n\in\mathfrak{X}^{\mathbb{C}}(M)\;\text{generate}\;
\mathfrak{Z}(M),\\
 \label{eq:essa}
\mathbb{E}_{\mathfrak{Z}}(M)\supset\overline{\mathfrak{Z}(M)},\\
\label{eq:essd}
L_0\;\text{is a real vector field},\\
\label{eq:esse}
L_{n+1}\in\mathfrak{Z}(M)+\overline{\mathfrak{Z}(M)},\\
\label{eq:essi}
a\in\mathcal{C}^{\infty}(M).
\end{gather}
Let us define:
\begin{gather}
  \label{eq:essf}
\mathbb{A}'(M)=\mathbb{A}'_{(0)}(M)=
\mathbb{A}_{\mathfrak{Z}}(M)+\mathcal{C}^{\infty}(M,\mathbb{R})\,L_0\\
\label{eq:essg}
\mathbb{A}'_{(h)}(M)=[\mathbb{A}'(M),\mathbb{A}'_{(h-1)}(M)]\quad
\text{for}\;h\geq{1},\\
\label{eq:essh}
\mathfrak{T}''(M)=\sum_{h=0}^{\infty}\mathbb{A}'_{(h)}
\end{gather}
Then the second order differential operator
\begin{equation}
  \label{eq:essb}
  P(u)=L_0(u)+\sum_{j=1}^n\bar{L}_jL_j(u)+L_{n+1}(u)+a\,u
\end{equation}
is hypoelliptic at all points $p\in{M}$ where
\begin{equation}
  \label{eq:essj}
  \{X(p)\mid X\in\mathfrak{T}''(M)\}=T_pM.
\end{equation}
\end{thm}
\begin{proof} We shall prove that, for every $X\in\mathfrak{T}''(M)$,
  \begin{equation}
    \label{eq:esb}\left\{
      \begin{gathered}
        \forall U^{\mathrm{open}}\Subset{M},\;
\exists\epsilon>0,\;\exists{C}>0,\;\text{s.t.}\\
\|X(u)\|_{\epsilon-1}^2\leq C\left(\|(P(u)\|^2_0+
\|u\|_0^2\right),\quad\forall u\in\mathcal{C}^{\infty}_0(U).
      \end{gathered}
    \right.
  \end{equation}
We observe that the proof of Theorem \ref{tm:cpd} shows that, if
$Y\in\mathfrak{X}(M)$ satisfies
\eqref{eq:esb}, and $X\in\mathbb{A}_{\mathfrak{Z}}(M)$, then $[X,Y]$
also satisfies \eqref{eq:esb} (with $\epsilon/2$ substituting $\epsilon$). 
Thus to show
that all $X\in\mathfrak{T}''(M)$ satisfy \eqref{eq:esb}, it suffices
to prove the following
\begin{lem}\label{lm:bi}
If $Y\in\mathfrak{X}(M)$ satisfies \eqref{eq:esb}, then
$[L_0,Y]$ also satisfies \eqref{eq:esb}.  
\end{lem}
\begin{proof} We closely follow the argument in \cite[p.66-68]{K71}.
Condition \eqref{eq:essa} means that, for every $U\Subset{M}$
there is a constant $C_0>0$ such that
\begin{equation}
  \label{eq:esa}
  \sum_{j=1}^n\|\bar{L}_j(u)\|^2_0\leq{C}_0\left(\sum_{j=1}^n\|L_j(u)\|^2_0+
\|u\|_0^2\right),\quad\forall u\in\mathcal{C}^{\infty}_0(U).
\end{equation}
Thus, by \eqref{eq:cpg}, we obtain, with a constant $C_1$ that only
depends on $U\Subset{M}$,
\begin{equation}
  \label{eq:esc}\left\{\begin{aligned}
  \sum_{j=1}^n\left(\|L_j(u)\|^2_0+\|\bar{L}_j(u)\|^2_0\right)
\leq C_1\left(|(P(u)|u)_0|+\|u\|^2_0\right),\qquad\\
\quad\forall u\in\mathcal{C}^{\infty}_0(U).
\end{aligned}\right.
\end{equation}
Let $Y\in\mathfrak{X}(M)$ satisfy \eqref{eq:esb}. We have,
with $T_{2\delta-1}\in\Psi^{2\delta-1}(U)$, and for all
$u\in\mathcal{C}^{\infty}_0(U)$,
\begin{equation*}
  \begin{aligned}
    \|[L_0,Y](u)\|_{\delta-1}^2&\leq C_2\left(
|([L_0,Y](u)|T_{2\delta-1}(u))_0|+\|u\|_0^2\right)\\
&\leq C_2\left(
|(L_0Y(u)|T_{2\delta-1}(u))_0|+
|(YL_0(u)|T_{2\delta-1}(u))_0|+
\|u\|_0^2\right).
  \end{aligned}
\end{equation*}
We shall estimate separately each summand inside the parentheses 
in the last term. In the following $U$ will be a relatively compact
open subset of $M$ and all functions $u$ will be smooth
and have compact support in some 
fixed relatively compact open subset of $U$.
\par
We have 
\begin{equation*}P^*(u)=-L_0(u)+\sum_{j=1}^nL_j^*L_j(u)
+L'_{n+1}(u)+a'u,
\end{equation*}
with $L'_{n+1}\in\mathfrak{Z}(M)+\overline{\mathfrak{Z}(M)}$
and $a'\in\mathcal{C}^{\infty}(M)$. Hence using
\eqref{eq:esc},
  \begin{equation}\tag{$*$}\label{eq:esxx}
\begin{aligned}
\left|(L_0Y(u)|T_{2\delta-1}(u))_0\right|
\leq\left|(P^*Y(u)|T_{2\delta-1}(u))_0\right|\qquad\qquad\qquad
\qquad\qquad\\
+
\sum_{j=1}^n\left|(L_j^*L_jY(u)|T_{2\delta-1}(u))_0\right|
+C_3(\|u\|_0^2+\|Y(u)\|_{2\delta-1}^2+\|P(u)\|_0^2).
\end{aligned}
\end{equation}
The term in parenthesis can be estimated by a constant times
$(\|u\|^2_0+\|P(u)\|_0^2)$, provided we choose a $\delta$ so small
that $Y$ satisfies \eqref{eq:esb} in $U$ with some $\epsilon\geq{2\delta}$.
For the first summand on the right hand side of \eqref{eq:esxx}, we have
\begin{equation*}
(P^*Y(u)|T_{2\delta-1}(u))_0=(T_{2\delta-1}^*(Y(u))|P(u))_0+
(Y(u)|[P,T_{2\delta-1}](u))_0.  
\end{equation*}
We obtain
\begin{equation*}
  |(T_{2\delta-1}^*(Y(u))|P(u))_0|\leq C_4\left(\|Y(u)\|_{2\delta-1}^2+
\|P(u)\|^2_0\right).
\end{equation*}
The right hand side is bounded 
by a constant times $(\|u\|_0^2+\|P(u)\|_0^2)$, provided again that
 \eqref{eq:esb} holds for $Y$ with $\epsilon\geq{2\delta}$.
\par  For the commutator $[P,T_{2\delta-1}]$, we have
\begin{equation*}
  [P,T_{2\delta-1}]=\sum_{j=1}^nT'_{2\delta-1}L_j^*(u)+T''_{2\delta-1}L_j(u)
+T'''_{2\delta-1},
\end{equation*}
with $T'_{2\delta-1},T''_{2\delta-1},T'''_{2\delta-1}\in\Psi^{2\delta-1}(U)$.

Thus, because of \eqref{eq:esc}, 
also the term $|(Y(u)|[P,T_{2\delta-1}](u))_0|$ is
 bounded 
by a constant times $(\|u\|_0^2+\|P(u)\|_0^2)$, provided again that
 \eqref{eq:esb} holds for $Y$ with $\epsilon\geq{2\delta}$.\par
We have
\begin{equation*}
  T^*_{2\delta-1}L_j^*L_j=L_j^*L_jT^*_{2\delta-1}+
L_jB_{2\delta-1}'+L_j^*B_{2\delta-1}''+B_{2\delta-1}'''
\end{equation*}
with $B_{2\delta-1}',B_{2\delta-1}'',B_{2\delta-1}'''\in\Psi^{2\delta-1}(U)$.
Thus with a constant $C_4>0$,
\begin{equation*}\begin{aligned}
  \left|(L_j^*L_jY(u)|T_{2\delta-1}(u))_0\right|&\leq
\left|(L_jT^*_{2\delta-1}(u)|L_j(u))_0\right|+
\left|(B_{2\delta-1}'Y(u)|L_j^*(u))_0\right|\\
&
+\left|(B_{2\delta-1}''Y(u)|L_j(u))_0\right|
+\left|(B_{2\delta-1}'''Y(u)|u)_0\right|\\
&\leq
C_4\left(\|L_jT^*_{2\delta-1}Y(u)\|_0^2+\|L_j(u)\|^2_0+\|\bar{L}_j(u)\|^2_0
\right.\\ &\qquad \qquad\qquad \qquad\qquad \left. +
\|Y(u)\|_{2\delta-1}^2+\|u\|_0^2\right).
\end{aligned}
\end{equation*}
Therefore,
provided again that \eqref{eq:esb} holds for $Y$ with some
$\epsilon\geq{2\delta}$, all terms but those of the form
$\|L_jT^*_{2\delta-1}(u)\|_0^2$ are bounded by a constant times
$(\|u\|_0^2+\|P(u)\|_0^2)$. Thus 
we only need
to bound the terms $\|L_jT^*_{2\delta-1}(u)\|_0^2$. We have, by 
\eqref{eq:cpg}, with some constant $C_5>0$,
\begin{equation*}
 \sum_{j=1}^n \|L_jT^*_{2\delta-1}Y(u)\|_0^2\leq C_5\left(
\left|(P(T^*_{2\delta-1}Y(u))|T^*_{2\delta-1}Y(u))_0\right|+
\|T^*_{2\delta-1}Y(u)\|_0^2\right).
\end{equation*}
The last summand inside
the parentheses on the right hand side is bounded by a constant
times
$(\|u\|_0^2+\|P(u)\|_0^2)$,
provided again that \eqref{eq:esb} holds for $Y$ with some
$\epsilon\geq{2\delta}$. Let us consider the first one.
Note that the composition
$T_{2\delta}=T^*_{2\delta-1}\circ{Y}$ is a
pseudodifferential
operator in 
$\Psi^{2\delta}(U)$. We have the 
commutation formula
\begin{equation*}
  [P,T_{2\delta}]=\sum_{j=1}^n(F'_{j_{2\delta}}L_j+F''_{j_{2\delta}}\bar{L}_j)+
F'''_{{2\delta}}
\end{equation*}
with $F'_{j_{2\delta}},F''_{j_{2\delta}},F'''_{{2\delta}}\in\Psi^{2\delta}(U)$.
\par
Thus
\begin{equation*}
 \begin{aligned}
  (P(T^*_{2\delta-1}Y(u))|T^*_{2\delta-1}Y(u))_0=
(P(T_{2\delta}(u))|T^*_{2\delta-1}Y(u))_0\qquad\qquad\qquad\\
=(P(u)|T^*_{2\delta}T^*_{2\delta-1}Y(u))_0+
\sum_{j=1}^n(L_j(u)|[F'_{j_{2\delta}}]^*T^*_{2\delta-1}Y(u))_0\\
+\sum_{j=1}^n(\bar{L}_j(u)|[F''_{j_{2\delta}}]^*T^*_{2\delta-1}Y(u))_0
+(u|[F'''_{{2\delta}}]^*T^*_{2\delta-1}Y(u))_0.
\end{aligned}
\end{equation*}
Hence we obtain, with some constant $C_6>0$,
\begin{equation*}\begin{aligned}
 \|L_jT^*_{2\delta-1}Y(u)\|_0^2\leq C_6\big(
\|P(u)\|_0\|Y(u)\|_{4\delta-1}+\sum_{j=1}^n\left(\|L_j(u)\|^2_0+
\|\bar{L}_j(u)\|^2_0\right)\\
+\|Y(u)\|_{4\delta-1}^2+\|u\|_0^2\big),
 \end{aligned}
\end{equation*}
which can be bounded by the right hand side of \eqref{eq:esb},
provided $Y$ satisfies \eqref{eq:esb} with $\epsilon\geq{4\delta}$.
\par
Finally we note that
\begin{equation}
\label{eq:esyyy}
  \begin{aligned}
    |(YL_0(u)|T_{2\delta-1}(u))_0|\leq
C_7(|(T^*_{2\delta-1}(u)|L_0Y(u))_0|+\|L_0(u)\|_{2\delta-1}^2\\+
\|Y(u)\|_{2\delta-1}^2+\|u\|_0^2).
  \end{aligned}
\end{equation}
Thus, by repeating the discussion above with $T^*_{2\delta-1}$
replacing $T_{2\delta-1}$, we find that the left hand side of
\eqref{eq:esyyy},
provided $Y$ satisfies \eqref{eq:esb} with $\epsilon\geq{4\delta}$,
 is bounded by a constant times $(\|u\|_0^2+\|P(u)\|_0^2)$.
This concludes the proof of the Lemma.
\end{proof}
\subsubsection*{End of the Proof of Theorem \ref{tm:bh}}
By the discussion at the beginning of the proof, and by Lemma \ref{lm:bi},
we obtain that for every relatively compact open subset $U$ of $M$,
which is contained in the open subset $M'$ of $M$, consisting of the points $p$
where \eqref{eq:essj} is satisfied, there are positive constants
$\epsilon>0$ and $c_0>0$ such that
\begin{equation}
  \label{eq:estb}\left\{\begin{aligned}
  \|u\|_{\epsilon}^2+\sum_{j=1}^n\left(\|L_j(u)\|_0^2+\|\bar{L}_j(u)\|_0^2\right)
\leq{c}_0\left(\|u\|_0^2+\|P(u)\|_0^2\right),\\
  \quad\forall u\in\mathcal{C}^{\infty}_0(U).
\end{aligned}\right.
\end{equation}
One easily shows by recurrence that, 
if $U$ is a relatively compact open subset of $M$, then
there is a positive constant $\epsilon>0$, and, 
for every real $s\geq{0}$ another constant $c_s\geq{0}$, such that
\begin{equation}
  \label{eq:esta}\left\{\begin{aligned}
  \|u\|_{\epsilon+s}^2+\sum_{j=1}^n\left(\|L_j(u)\|_s^2+
\|\bar{L}_j(u)\|_s^2\right)
\leq{c}_0\left(\|u\|_0^2+\|P(u)\|_s^2\right),\quad\\
  \quad\forall u\in\mathcal{C}^{\infty}_0(U).
\end{aligned}\right.
\end{equation}
The hypoellipticity of $P$ in $U$, with a gain of $\epsilon$
derivatives,  
follows in a standard way from
\eqref{eq:esta} (see e.g. \cite{Hor67}).
\end{proof}

\section{Applications to almost \texorpdfstring{$CR$}{CR} manifolds}
\label{sec:examp}
In this section we shall consider the case where
$M$ is an almost $CR$ manifold of $CR$ dimension $n$ and $CR$ codimension $k$,
and $\mathfrak{Z}(M)$ is the distribution of vector fields
of type $(0,1)$ on $M$. 
This means that conditions $(i)$, $(ii)$ and $(iii)$ below are satisfied:
\begin{align} 
\tag{$i$}&M \;\text{has real dimension}\; 2n+k,\\
\tag{$ii$}& \mathfrak{Z}(M)\;\text{has constant rank}\; n,\\
\tag{$iii$}& \mathfrak{Z}(M)\cap\overline{\mathfrak{Z}(M)}=\{0\},\\
\tag{$iv$}& [\mathfrak{Z}(M),\mathfrak{Z}(M)]\subset\mathfrak{Z}(M).
\end{align}
When the formal integrability 
condition $(iv)$ is also satisfied, we say that $M$ is a
$CR$ manifold.\par
When $M$ is an almost $CR$ manifold, it is customary to write
$T^{0,1}M$ for the complex bundle with fibers $T^{0,1}_pM=Z_pM$.
The $Z$-morphisms of \S\ref{sec:map} are then the $CR$ maps, i.e. the
smooth maps $\phi:N\to{M}$ with 
$d\phi^{\mathbb{C}}(T^{0,1}_qN)\subset T^{0,1}_{\phi(q)}M$ for all $q\in{N}$.
\begin{lem}
  \label{lm:dc}
  If an almost CR manifold $M$ satisfies the higher
  order Levi form concavity condition at a point $p$, then $M$ is of
  finite type at $p$.\qed
\end{lem}
The next statements clarify in what sense condition \eqref{eq:cra}
is a pseudoconcavity condition.
\begin{prop}
  \label{pp:db}
  Let $M$ be a CR manifold of hypersurface type, i.e. of
$CR$ codimension $k=1$.  If $\mathfrak{Z}(M)$
  satisfies the higher Levi form concavity condition 
\eqref{eq:cra}
at a point $p$,
which is 
  regular for $\mathfrak{Z}(M)$, then $M$ is strictly pseudoconcave at
  $p$.
\end{prop}
\begin{proof}
  Let $0\neq\xi\in H^0_pM$. We want to prove that $\mathcal{L}_{\xi}$
is indefinite. Assume by contradiction that 
this is not the case. Replacing, if needed, $\xi$ by $(-\xi)$, we can
assume that
$\mathcal{L}_{\xi}\geq{0}$. By assumption \eqref{eq:cra} we can
choose $Z_0\in\mathfrak{Z}(M)$ and
  $Z_1,\dots,Z_r\in\mathbb{K}_{\mathfrak{Z}}(M)
\cup\overline{\mathbb{K}_{\mathfrak{Z}}(M)}$ satisfying
\begin{equation}\tag{$*$}\label{eq:starq}
  \xi([Z_1,\dots,Z_r,\bar{Z}_0])\neq{0}.  
\end{equation}
We can take
$r$ minimal with this property. In particular, we
  have 
\begin{equation}\tag{$**$}\label{eq:starqq}
[Z_2,\dots,Z_r,\bar{Z}_0](p)=Z(p)+\overline{W}(p),
\end{equation}
with
  $Z,W\in\mathfrak{Z}(M)$. Assume that $Z_1\in\mathbb{K}_{\mathfrak{Z}}(M)$.
Then $\xi([Z_1,Z])=0$, because of the integrability condition $(iv)$,
 and then, from \eqref{eq:starq},
we have
  $\mathcal{L}_\xi(Z_1,\overline{W})\neq 0$ and
  $\mathcal{L}_\xi(Z_1,\overline{Z}_1)=0$,
yielding a contradiction.
Likewise, if $Z_1\in\overline{\mathbb{K}_{\mathfrak{Z}}(M)}$,
we have ${\xi}([Z_1,\bar{W}])=0$ by $(iv)$, 
and hence, from
\eqref{eq:starq}, $\mathcal{L}_{\xi}(Z,Z_1)\neq{0}$, while
$\mathcal{L}_{\xi}(\bar{Z}_1,Z_1)={0}$, 
contradicting the assumption that $\mathcal{L}_{\xi}\geq{0}$. The proof
is complete.
\end{proof}

\begin{cor}
  \label{cor:dda}
  Let $N$ be a generic CR submanifold of a CR manifold $M$
(this means that 
$T^{0,1}_pN=T^{0,1}_pM\cap{T}^{\mathbb{C}}_pN$ 
and the restriction map $H^0_pM\to{H}^0_pN$ is injective
for all $p\in{N}\subset{M}$). 
If $M$ is of hypersurface type and
$\mathfrak{Z}(N)$ satisfies the higher Levi form concavity
  condition \eqref{eq:cra}
at a point $p_0$, regular  for 
$\mathfrak{Z}(M)$, then $M$ is
  strictly pseudoconcave at~$p_0$.
\end{cor}
\begin{proof}
  The statement follows from Propositions~\ref{pp:da} and \ref{pp:db}.
\end{proof}
\begin{lem}
  \label{lm:ddb}
  Let $M,N$ be almost $CR$ manifolds, 
and $\varpi\colon N\to M$ a $CR$ map and a
  smooth submersion. 
We denote by $\mathfrak{Z}(M)$ and $\mathfrak{Z}(N)$ the distributions
of $(0,1)$ vector fields on $M$, $N$, respectively.
If $M$ is strictly pseudoconvex at a point $p_0$, i.e. if there is
$\xi_0\in{H}^0_{p_0}M$ with $\mathcal{L}_{\xi_0}>0$, and moreover
 $p_0$ is regular for $\mathfrak{Z}(M)$,
  then the higher Levi form concavity
  condition \eqref{eq:cra}
for $\mathfrak{Z}(N)$ is not satisfied at any point $q_0\in\varpi^{-1}(p_0)$.
\end{lem}
\begin{proof} Replacing $M$ by an open neighborhood of $p_0$ in $M$,
we can assume that $M$ is strictly pseudoconvex at all points.
Then $\mathbb{K}_{\mathfrak{Z}}(M)=0$, and hence 
$\mathbb{K}_{\mathfrak{Z}}(N)$ is contained in the complexification of
the vertical distribution $\mathfrak{V}^{\varpi}(N)$. 
Indeed, if $\xi\in{H}^{\oplus}_pM$ and $q\in\varpi^{-1}(p)$,
the pullback $\varpi^*(q)(\xi)$ belongs to ${H}^{\oplus}_qN$.\par
Thus if $\eta_0=\varpi^*(q_0)(\xi_0)$ for $q_0\in\varpi^{-1}(p_0)$,
then $\eta_0$ does not satisfy \eqref{eq:cra}.
Indeed, $\eta_0$ vanishes on the pullback of 
$\mathfrak{Z}(M)+\overline{\mathfrak{Z}(M)}$, and
this distribution is a $\mathfrak{V}^{\varpi}(N)$-module.
\end{proof}
Theorem \ref{tm:cca} and Corollary \ref{cor:Db} yield
\begin{thm}
Let $M$ be an almost $CR$ manifold and 
assume that condition \eqref{eq:be} is satisfied
by the distribution $\mathfrak{Z}(M)$ of its $(0,1)$ vector
fields at all points of $M$. Then
\begin{enumerate}
\item If
$E\xrightarrow{\pi}M$ a 
Hermitian vector bundle on $M$, 
endowed with a $\mathbb{C}$-linear connection $\nabla$, then all 
weak solutions $u\in{L}^2_{\mathrm{loc}}(M,E)$ of 
\begin{equation*}
\nabla_Z(u)\in\mathcal{C}^{\infty}(M,E),\quad\forall{Z}\in\mathfrak{Z}(M)  
\end{equation*}
are smooth sections of $E\xrightarrow{\pi}M$. \par
In particular, all $CR$ sections of $E$ \textup(i.e. weak $L^2_{\mathrm{loc}}$
solutions of $\nabla_Z(u)=0$, for all $Z\in\mathfrak{Z}(M)$\textup) are smooth.
\item
In case $M$ is compact, the space of $CR$ sections of $E$ is a finite
dimensional $\mathbb{C}$-linear space. 
\item
If 
$M$ is a compact $CR$ manifold, then the cohomology groups
$H^{p,1}_{\bar\partial_M}(M)$, for $p=1,\hdots,n+k$, 
of the tangential Cauchy-Riemann complexes, are Hausdorff. \qed
\end{enumerate}
\end{thm}

\section{Subellipticity conditions for 
homogeneous 
\texorpdfstring{$CR$}{CR} manifolds}\label{sec:hom}
Let $M$ be a $CR$ manifold, homogeneous for the $CR$ action of a
Lie group~$\mathbf{G}$. 
Fix a base point $\mathbf{o}\in{M}$ and denote by
\begin{equation}
  \label{eq:pbbaa}
  \varpi:\mathbf{G}\ni{g}\to{g}\cdot\mathbf{o}\in{M}
\end{equation}
the associated principal bundle.
In \cite{MN05} we associated to $M$ and the base point $\mathbf{o}$
the $CR$ algebra $(\mathfrak{g},\mathfrak{q})$, that is the pair
consisting of
\begin{itemize}
  \item[-] the Lie algebra $\mathfrak{g}$ of $\mathbf{G}$,
\item[-] the complex Lie subalgebra $\mathfrak{q}$ of 
$\mathfrak{g}^{\mathbb{C}}=\mathfrak{g}+i\mathfrak{g}$ given by
\begin{equation}
  \label{eq:pbba}
  \mathfrak{q}=[d\varpi_{e}]_{{}_\mathbb{C}}^{-1}(T^{0,1}_{\mathbf{o}}M),
\end{equation}
where 
$[d\varpi_e]_{{}_\mathbb{C}}:
\mathfrak{g}^{\mathbb{C}}\to{T}^{\mathbb{C}}_{\mathbf{o}}M$
is the complexification of the differential 
$d\varpi_e:\mathfrak{g}\to{T}_{\mathbf{o}}M$ of $\varpi$ at the identity.
\end{itemize}
For $X\in\mathfrak{g}$ we denote by $X^*$ the left-invariant vector
field in $\mathbf{G}$ with $X^*(e)=X$. Let
\begin{align}
  \label{eq:pbbb}
\mathfrak{q}^*&=\{X^*+iY^*\mid X,Y\in\mathfrak{g},\;X+iY\in\mathfrak{q}\},\\
\label{eq:pbbc}
\mathfrak{Z}(\mathbf{G})&=\mathcal{E}(M)\otimes
\mathfrak{q}^*=\text{the vector distribution spanned by $\mathfrak{q}^*$}.
\end{align}
The following statement is straightforward.
\begin{prop}\label{pp:pba}
Let $M=\mathbf{G}/\mathbf{G}_{\mathbf{o}}$ be a $\mathbf{G}$-homogeneous
$CR$ manifold, $\mathfrak{Z}(M)$ the distribution of
$(0,1)$-vector fields on $M$, and $\mathfrak{Z}(\mathbf{G})$
the distribution on $\mathbf{G}$ defined by \eqref{eq:pbbc}. Then
\begin{enumerate}
\item The distributions $\mathfrak{Z}(M)$ and 
$\mathfrak{Z}(\mathbf{G})$ are regular at all points;
\item the principal bundle fibration $\mathbf{G}\xrightarrow{\varpi}M$
is a $Z$-morphism,
\begin{equation}\label{eq:pbbd}
\mathfrak{Z}(\mathbf{G})=[\varpi^*\mathfrak{Z}](\mathbf{G})
\end{equation}  
and conditions $(2)$, $(3)$, $(4)$ 
of Proposition \ref{pp:da} are satisfied.\qed
\end{enumerate}
\end{prop}
Proposition \ref{pp:pba} can be used to reduce
the question of the subellipticity
of the distribution of $(0,1)$-vector fields of a
homogeneous $CR$ manifold to
Lie algebra computations. \par
Let $(\mathfrak{g},\mathfrak{q})$ be the $CR$ algebra associated
with the $\mathbf{G}$-homogeneous $CR$ manifold $M$ and its base
point $\mathbf{o}$. Let
\begin{equation}
  \label{eq:pbbe}
  \mathfrak{t}^0=H^0_e\mathbf{G}=
\{\xi\in\mathfrak{g}^*\mid \xi(\mathrm{Re}(Z))=0,\;
\forall{Z}\in\mathfrak{q}\}.
\end{equation}
To each $\xi\in\mathfrak{t}^0$ we associate the Levi form
\begin{equation}
  \label{eq:pbbg}
\mathcal{L}_{\xi}(Z,\bar{W})=i\xi([Z,\bar{W}]),\;
\text{for}\; Z,W\in\mathfrak{q}.  
\end{equation}
We also set
\begin{align}
  \label{eq:Faa}
  \mathfrak{t}^{\oplus}&=\{\xi\in\mathfrak{t}^0\mid \mathcal{L}_{\xi}\geq{0}\},\\
\label{eq:Fab}
\mathfrak{k}_{\mathfrak{q}}&=\{Z\in\mathfrak{q}\mid \mathcal{L}_{\xi}(Z,\bar{Z})=0,
\;\forall\xi\in\mathfrak{t}^{\oplus}\}.
\end{align}
Since all points of $M$ and of $\mathbf{G}$ are regular for
$\mathfrak{Z}(M)$ and $\mathfrak{Z}(\mathbf{G})$, respectively, we obtain
from Proposition \ref{pp:pcj}:
\begin{lem}\label{lem:Fab}
Let $(\mathfrak{g},\mathfrak{q})$ be a $CR$ algebra. 
The set
$\mathfrak{k}_{\mathfrak{q}}$
is a linear subspace of $\mathfrak{q}$ and is equal to the set
\begin{equation*}
  \left\{Z_0\in\mathfrak{q}\left|\exists Z_1,\hdots,Z_r\in\mathfrak{q},\;\text{s.t.}\;
\sum_{j=0}^r[Z_j,\bar{Z}_j]\in\mathfrak{q}+\bar{\mathfrak{q}}\right\}\right. .
\end{equation*}
The elements 
\begin{equation*}
Z^*=(\mathrm{Re}\,Z)^*+i(\mathrm{Im}\,Z)^*,\quad\text{for
$Z\in\mathfrak{k}_{\mathfrak{q}}$},
\end{equation*}
generate the distribution
$\mathbb{K}_{\mathfrak{Z}}(\mathbf{G})$, that is equal to 
$\Theta_{\mathfrak{Z}}(\mathbf{G})$. \qed
\end{lem}
As a consequence of Propositions \ref{pp:crx}, \ref{pp:pba} 
and Lemma \ref{lem:Fab} we have
\begin{prop}\label{pp:pbc}
Let $M=\mathbf{G}/\mathbf{G}_{\mathbf{o}}$ be a homogeneous $CR$ manifold
and let $(\mathfrak{g},\mathfrak{q})$ be the $CR$ algebra associated with
$M$ and the base point $\mathbf{o}$. 
We denote by $\mathfrak{Z}(M)$ the distribution of $(0,1)$-vector fields
on $M$.
Then the following are equivalent.
\begin{enumerate}
\item $\mathfrak{Z}(M)$ satisfies condition \eqref{eq:cra}.
\item $\mathfrak{Z}(M)$ satisfies condition \eqref{eq:be}.
\item $\mathfrak{Z}(\mathbf{G})$ 
 satisfies condition \eqref{eq:cra}.
\item $\mathfrak{Z}(\mathbf{G})$ 
 satisfies condition \eqref{eq:be}.
\item $(\mathfrak{g},\mathfrak{q})$ satisfies the condition:
\begin{equation}
  \label{eq:pbbh} \left\{\begin{gathered}
\forall\xi\in\mathfrak{t}^{\oplus}\setminus\{0\},\quad
\exists Z_0\in\mathfrak{q}\;\text{and}\; 
Z_1,\hdots,Z_r\in\mathfrak{k}_{\mathfrak{q}}\cup\bar{\mathfrak{k}}_{\mathfrak{q}}\\
\text{s.t.}\quad
 i\xi([Z_1,\hdots,Z_r,\bar{Z}_0])\neq{0}.
\end{gathered}\right.\begin{gathered} \quad\\
\quad \qed
\end{gathered}
\end{equation}
\end{enumerate}
\end{prop}
\section{Orbits of a real form in a complex flag manifold}
\label{sec:orb}
In this section we investigate the subellipticity of the 
distribution of the $(0,1)$-vector fields of the homogeneous 
$CR$ manifolds which are real orbits of 
real forms in complex
flag manifolds. The study of their $CR$ geometry has been 
already started
in \cite{AMN06, AMN07},
to which we refer for the complete explanation of many
details.\par
\subsection{Complex flag manifolds and orbits of a real form}
We recall that a \emph{complex flag manifold} is a 
closed complex projective variety, that is a
coset space
of a connected semisimple complex Lie group
$\mathbf{G}^{\mathbb{C}}$ 
with respect to 
a complex \emph{parabolic} subgroup.
\par
A \emph{real form} $\mathbf{G}$ of $\mathbf{G}^{\mathbb{C}}$ is
a \emph{real}
Lie subgroup of $\mathbf{G}^{\mathbb{C}}$ 
whose Lie algebra $\mathfrak{g}$ is a real form of
$\mathfrak{g}^{\mathbb{C}}$, i.e. such that $\mathfrak{g}\subset
\mathfrak{g}^{\mathbb{C}}$ and $\mathfrak{g}^{\mathbb{C}}=\mathfrak{g}+
i\mathfrak{g}$. 
We shall write $\bar{Z}$ for the conjugate of an element 
$Z\in\mathfrak{g}^{\mathbb{C}}$ with respect to the real form $\mathfrak{g}$.
The left action of $\mathbf{G}$ 
decomposes $F$ 
into
a finite set of $\mathbf{G}$-orbits (see \cite{Wolf69}).
All $\mathbf{G}$-orbits $M=\mathbf{G}\cdot\mathbf{o}$ are
generically embedded $CR$ submanifolds of $F$.
It turns out that
$F$ and 
also $M$, if we take $\mathbf{G}$ connected, 
are completely determined
by the Lie algebra $\mathfrak{g}^{\mathbb{C}}$, 
by its real form
$\mathfrak{g}$, and by the complex parabolic Lie subalgebra 
$\mathfrak{q}\subset\mathfrak{g}^{\mathbb{C}}$ of the isotropy
subgroup $\mathbf{Q}$ at $\mathbf{o}$.
Thus we shall write
$M=M(\mathfrak{g},\mathfrak{q})$, for the homogeneous $CR$ manifold
$M$. We denote by
$\mathbf{G}_{\mathbf{o}}=\mathbf{Q}\cap\mathbf{G}$ the isotropy
subgroup at the base point $\mathbf{o}$ 
and by $\mathbf{G}\xrightarrow[\mathbf{G}_{\mathbf{o}}]{\varpi}M$
the principal $\mathbf{G}_{\mathbf{o}}$-bundle 
$\varpi\colon\mathbf{G}\ni{g}\to{g}\cdot\mathbf{o}\in{M}$.
\par
The Lie algebra $\mathfrak{g}_{\mathbf{o}}$ of 
$\mathbf{G}_{\mathbf{o}}$ 
contains a Cartan subalgebra $\mathfrak{h}$
of $\mathfrak{g}$ (see e.g. \cite{Wolf69, AMN06, AMN07}).
Its complexification 
$\mathfrak{h}^{\mathbb{C}}=\mathfrak{h}+i\mathfrak{h}$ is a Cartan subalgebra
of $\mathfrak{g}^{\mathbb{C}}$. Let $\mathcal{R}$ be the set of roots
of $\mathfrak{g}^{\mathbb{C}}$ with respect to $\mathfrak{h}^{\mathbb{C}}$,
and 
$\mathfrak{g}^{\mathbb{C}}_{\alpha}=\{Z\in\mathfrak{g}^{\mathbb{C}}\mid
[H,Z]=\alpha(H)Z,\;\forall{H}\in\mathfrak{h}^{\mathbb{C}}\}$
the eigenspace corresponding to the root $\alpha\in\mathcal{R}$.
Since $\mathfrak{h}^{\mathbb{C}}\subset\mathfrak{q}$, the subalgebra
$\mathfrak{q}$ is 
$\mathrm{ad}_{\mathfrak{g}^{\mathbb{C}}}(\mathfrak{h}^{\mathbb{C}})$-invariant. 
Hence we have
\begin{align}\label{eq:mof}
  \mathfrak{q}&=\mathfrak{h}^{\mathbb{C}}\oplus\sum_{\alpha\in\mathcal{Q}}
\mathfrak{g}^{\mathbb{C}}_{\alpha},\quad\text{for}\\
\mathcal{Q}&=\{\alpha\in\mathcal{R}\mid \mathfrak{g}^{\mathbb{C}}_{\alpha}
\subset\mathfrak{q}\}.
\end{align}
To say that $\mathfrak{q}$ is parabolic means that
$\mathcal{Q}$ contains a positive system 
of roots $\mathcal{R}^+$. Let $\prec$ be a corresponding partial order
on the linear span $\mathfrak{h}^*_{\mathbb{R}}$ of $\mathcal{R}$, 
and
$\mathcal{B}=\{\alpha_1,\hdots,\alpha_{\ell}\}$ 
the set of simple positive roots in $\mathcal{R}^+$.
Every root $\alpha\in\mathcal{R}$ can be written in a unique way as
a linear combination with integral coefficients of the elements of
$\mathcal{B}$:
\begin{equation}
  \label{eq:moh}
  \alpha=\sum_{j=1}^{\ell}k_j\alpha_j,
\end{equation}
where all $k_j$'s are either $\geq{0}$, or $\leq{0}$, according to whether
$\alpha$ is positive or negative with respect to $\prec$.
We define the \emph{support} of $\alpha$ to be the set
\begin{equation}
  \label{eq:moi}
  \mathrm{supp}(\alpha)=\{\alpha_j\in\mathcal{B}\mid k_j\neq{0}\}.
\end{equation}
Let 
\begin{equation}
  \label{eq:moia}
  \Phi=\{\alpha\in\mathcal{B}\mid -\alpha\notin\mathcal{Q}\}
\end{equation}
be the set of simple roots $\alpha$ whose opposite $(-\alpha)$ does not belong to
$\mathcal{Q}$.
Then
\begin{equation}
  \label{eq:moj}
  \mathcal{Q}=\mathcal{R}^+\cup\{\alpha\prec{0}\mid\mathrm{supp}(\alpha)
\cap\Phi=\emptyset\}.
\end{equation}
Since $\mathcal{Q}$ is completely determined by $\Phi$, we shall write
$\mathcal{Q}_{\Phi}$, $\mathfrak{q}_{\Phi}$, $\mathbf{Q}_{\Phi}$
for the parabolic set of roots, the complex parabolic subalgebra and the
complex parabolic subgroup, 
respectively,
that are attached to any special choice of the
subset $\Phi$ of $\mathcal{B}$. We shall also introduce the notation
\begin{equation}
  \label{eq:mok}
  \begin{aligned}
    \mathcal{Q}^n_{\Phi}&=\{\alpha\succ{0}\mid\mathrm{supp}(\alpha)\cap\Phi
\neq\emptyset\}\\
\mathcal{Q}^r_{\Phi}&=\{\alpha\mid \alpha,-\alpha\in\mathcal{Q}\}
=\{\alpha\in\mathcal{R}\mid \mathrm{supp}(\alpha)\cap\Phi=\emptyset\}.
  \end{aligned}
\end{equation}
\par
The conjugation in $\mathfrak{g}^{\mathbb{C}}$ 
induced by the real form $\mathfrak{g}$
defines, by duality, a conjugation
$\alpha\to\bar{\alpha}$ in $\mathcal{R}$. 
We partition $\mathcal{R}$ into three subsets:
\begin{align}
\mathcal{R}_{\mathrm{re}}&=\{\alpha\in\mathcal{R}\mid\bar{\alpha}=\alpha\}
&&\text{real roots},\\
\mathcal{R}_{\mathrm{im}}&=\{\alpha\in\mathcal{R}\mid\bar{\alpha}=-\alpha\}
&&\text{imaginary roots},\\
\mathcal{R}_{\mathrm{cx}}&=\{\alpha\in\mathcal{R}\mid\bar{\alpha}\neq\pm\alpha\}
&&\text{complex roots}.
\end{align}
The Cartan subalgebra $\mathfrak{h}$ is invariant under a Cartan 
involution
$\vartheta$ of $\mathfrak{g}$, whose set of fixed points $\mathfrak{k}$
is a maximal compact Lie subalgebra of $\mathfrak{g}$. With
$\mathfrak{p}=\{X\mid \vartheta(X)=-X\}$, 
we have that for imaginary $\alpha$, the eigenspace
$\mathfrak{g}^{\mathbb{C}}_{\alpha}$ is contained either in
$\mathfrak{k}^{\mathbb{C}}=\mathfrak{k}+i\mathfrak{k}$, or
in $\mathfrak{p}^{\mathbb{C}}=\mathfrak{p}+i\mathfrak{p}$.
In the first case we say that $\alpha$ is a \emph{compact} root.
Compact roots form a root subsystem $\mathcal{R}_{\bullet}$ of 
$\mathcal{R}$.
\par
The $\mathbf{G}$-homogeneous $CR$ structure of $M$ is defined,
as in \S\ref{sec:hom}, by
assigning the subspace 
\begin{equation}
  \label{eq:mol}
  T^{0,1}_{\mathbf{o}}=d\varpi(e)(\mathfrak{q}).
\end{equation}
Here $e$ is the identity of $\mathbf{G}$ and
we denote by the same symbol $d\varpi(e)$
the complexification of the differential
$d\varpi(e):\mathfrak{g}=T_e\mathbf{G}\to{T}_{\mathbf{o}}M$. 

For each $Z\in\mathfrak{g}^{\mathbb{C}}$ we can consider the 
\textit{fundamental vector field} $Z^*\in\mathfrak{X}^{\mathbb{C}}(M)$.
Its real and imaginary parts are the infinitesimal generators of the
flows associated to the left translations by $\exp(t\mathrm{Re}(Z))$
and $\exp(t\mathrm{Im}(Z))$, respectively. 
\par
Fix a Chevalley basis\footnote{
This means that $Z_{\alpha}\in\mathfrak{g}^{\mathbb{C}}_{\alpha}$ for
all $\alpha\in\mathcal{R}$ and that, for $\alpha\in\mathcal{B}$, we also
have $[H_{\alpha},Z_{\alpha}]=2Z_{\alpha}$, 
$[H_{\alpha},Z_{-\alpha}]=-2Z_{-\alpha}$,
$[Z_{\alpha},Z_{-\alpha}]=-H_{\alpha}$; and moreover that the 
linear map defined by $H_{\alpha}\to -H_{\alpha}$ and
$Z_{\alpha}\to{Z}_{-\alpha}$ is an involutive automorphism
of the Lie algebra~$\mathfrak{g}^{\mathbb{C}}$.
Moreover, $Z_{\alpha}\in\mathfrak{g}$ when $\alpha\in\mathcal{R}_{\mathrm{re}}$
and,
if $\alpha\in\mathcal{R}_{\mathrm{im}}$,
we have $\bar{Z}_{\alpha}=Z_{-\alpha}$ when $\alpha$ is compact
and $\bar{Z}_{\alpha}=-Z_{-\alpha}$
when $\alpha$ is not compact. In general, $Z_{\bar\alpha}=
t_{\alpha}\bar{Z}_{\alpha}$, with $t_{\alpha}=\pm{1}$, 
for all roots $\alpha\in\mathcal{R}$, and we can take 
$t_{\alpha}=1$ when $\alpha\in\mathcal{R}_{\mathrm{re}}$ 
(see e.g. \cite{Bou75}).}
$\{H_{\alpha}\mid\alpha\in\mathcal{B}\}
\cup\{Z_{\alpha}\mid\alpha\in\mathcal{R}\}$ of $\mathfrak{g}^{\mathbb{C}}$.
The restrictions to $\mathbf{G}$ of the vector fields $H_{\alpha}$, for
$\alpha\in\mathcal{B}$, and $Z_{\alpha}$, for $\alpha\in\mathcal{Q}_{\Phi}$,
form a basis for the pullback $\mathfrak{Z}(\mathbf{G})$ of
$\mathfrak{Z}(M)$. 
\par
Since $\mathfrak{g}^{\mathbb{C}}$ is semisimple, the Killing form
$\kappa_{\mathfrak{g}^{\mathbb{C}}}$ is nondegenerate. Thus, 
as in \cite[\S 13]{AMN06}, we can use the Killing form to
identify the complexification 
$\mathfrak{t}^{0\,{\mathbb{C}}}$ of 
the space $\mathfrak{t}^0$ of \eqref{eq:pbbe} with the linear span
of
\begin{equation}
  \label{eq:moq}
  Z_{\alpha}\quad \text{for}\quad \alpha\in\mathcal{Q}_{\Phi}^n\cap
\bar{\mathcal{Q}}^n_{\Phi}.
\end{equation}
This is obtained by associating
to $Z_{\alpha}$ the linear form
\begin{equation}
  \label{eq:mor}
f_{\alpha}:  \mathfrak{g}^{\mathbb{C}}\ni{Z}\to
\kappa_{\mathfrak{g}^{\mathbb{C}}}(Z_{-\alpha},Z)=\mathrm{trace}
\left(\mathrm{ad}_{\mathfrak{g}^{\mathbb{C}}}(Z_{-\alpha})\circ
\mathrm{ad}_{\mathfrak{g}^{\mathbb{C}}}(Z)\right)\in\mathbb{C}.
\end{equation}
Correspondingly we obtain the complexified Levi forms
\begin{equation}
  \label{eq:mos}
  \mathbf{L}_{\alpha}(Z,\bar{W})=i f_{\alpha}([Z,\bar{W}]),\;
\text{for}\; Z,W\in\mathfrak{q}_{\Phi}.
\end{equation}
When $\alpha\in\mathcal{R}_{\mathrm{re}}\cap\mathcal{Q}$, it actually
corresponds to a Levi form \eqref{eq:sb} at $\mathbf{o}$. 
The intersection $\mathcal{Q}^n_{\Phi}\cap\bar{\mathcal{Q}}^n_{\Phi}$
does not contain imaginary roots. To a pair of complex roots 
$\alpha,\bar\alpha\in\mathcal{Q}^n_{\Phi}\cap\bar{\mathcal{Q}}^n_{\Phi}$,
correspond the two Hermitian symmetric forms obtained by polarization
from the Hermitian quadratic forms
\begin{equation*}
  \begin{aligned}
    \left[\mathrm{Re}\,\mathbf{L}_{\alpha}\right](Z,\bar{Z})&=
\frac{i}{2}\left(f_{\alpha}([Z,\bar{Z}])-\overline{f_{\alpha}([Z,\bar{Z}])}\right)
,\\
\left[\mathrm{Im}\,\mathbf{L}_{\alpha}\right](Z,\bar{Z})&=
\frac{1}{2}\left(f_{\alpha}([Z,\bar{Z}])+\overline{f_{\alpha}([Z,\bar{Z}])}\right)
.  \end{aligned}
\end{equation*}
Note that $L_{\alpha}(Z,\bar{W})=\pm\overline{{L}_{\bar\alpha}(W,\bar{Z})}$,
when $Z_{\bar\alpha}$ equals $\pm\bar{Z}_{\alpha}$, respectively.
Thus each Levi form can be written as a linear combination of the
$\mathbf{L}_{\beta}$'s for $\beta\in\mathcal{Q}_{\Phi}^n\cap
\bar{\mathcal{Q}}^n_{\Phi}$, the coefficients of $\mathbf{L}_{\beta}$
and $\mathbf{L}_{\bar{\beta}}$, for $\beta\in\mathcal{Q}_{\Phi}^n\cap
\bar{\mathcal{Q}}^n_{\Phi}\cap\mathcal{R}_{\mathrm{cx}}$ being either
conjugate or anticonjugate according to the sign in the equality
$Z_{\bar\beta}=\pm\bar{Z}_{\beta}$.
\par
\subsection{Semidefinite Levi forms}
\begin{lem}
  \label{lm:mmaa}
  Let $\mathcal{R}$ be a root system. There exist no triples
$\alpha,\beta,\gamma\in\mathcal{R}$ with
  \begin{equation*}\left\{
    \begin{gathered}
  \alpha+\bar{\alpha},\;\beta+\bar{\beta},\;\gamma+\bar{\gamma}\in\mathcal{R}, \\
   \alpha+\bar{\alpha}\neq\beta+\bar{\beta},\quad
\alpha+\bar{\beta}=\gamma+\bar{\gamma}.
    \end{gathered}\right.
  \end{equation*}
 \end{lem}
\begin{proof} Note that $\alpha,\;\beta,\;\gamma$ belong to the same
irreducible component of $\mathcal{R}$. Thus we may as well assume that
$\mathcal{R}$ is irreducible.
  We argue
by contradiction. \par
   The
  root subsystem $\mathcal{R}'$ generated by
  $\{\alpha+\bar{\alpha},\beta+\bar{\beta},\gamma+\bar{\gamma}\}$ is of type
  $\mathrm{B}_2$. 
Indeed, it is not of type $\mathrm{G}_2$ because it 
  is a proper subsystem of a larger root system, 
since all roots in $\mathcal{R}'$ are real, while 
the fact that $\alpha+\bar{\alpha}$ is a root implies that 
$\mathcal{R}_{\mathrm{cx}}$ is not empty.
Moreover, from
  $(\alpha+\bar{\alpha})+(\beta+\bar{\beta})=2(\gamma+\bar{\gamma})$,
  we deduce that $\mathcal{R}'$ has rank $2$ and contains roots
  of different lengths.
  Thus we can choose a basis $\{\mathrm{e}_j\}$ in
  $\mathfrak{h}_{\mathbb{R}}^*$ such that
  \[
  \alpha+\bar{\alpha}=2\mathrm{e}_1,\quad
  \beta+\bar{\beta}=2\mathrm{e}_2,\quad
  \gamma+\bar{\gamma}=\mathrm{e}_1+\mathrm{e}_2,
  \]
and, furthermore, that 
  \[
  \alpha=\mathrm{e}_1+a\mathrm{e}_3,\quad
  \beta=\mathrm{e}_2+a\mathrm{e}_3,\quad 
  \gamma=(1/2)\mathrm{e}_1+(1/2)\mathrm{e}_2+b\mathrm{e}_3+c\mathrm{e}_4
  \]
  for some $a,b,c\in\mathbb{R}$, and $a>0$.

  Let $\langle\alpha|\beta\rangle=
  2(\alpha|\beta)/(\alpha|\alpha)$. Recall that, if
  $\alpha_1,\alpha_2\in\mathcal{R}$ are not proportional, then
  $\left(\langle\alpha_1|\alpha_2\rangle\langle\alpha_2|\alpha_1\rangle\right)
\in\{0,1,2,3\}$. 
  Since
  \[
  \langle\alpha|\gamma+\bar{\gamma}\rangle
  \langle\gamma+\bar{\gamma}|\alpha\rangle = \frac{2}{1+a^2},
  \]
we have
$a=1$, that is
  $\alpha=\mathrm{e}_1+\mathrm{e}_3$ and
  $\beta=\mathrm{e}_2+\mathrm{e}_3$.
  Then
  \[
  \langle\gamma|\alpha+\bar{\alpha}\rangle
  \langle\alpha+\bar{\alpha}|\gamma\rangle= \frac{1}{(1/2)+b^2+c^2}
  \]
 implies that $b^2+c^2=1/2$. Hence $\Vert\gamma\Vert^2=1$. This gives
a contradiction,
because $\Vert\alpha\Vert^2=2$ and 
  $\Vert\alpha+\bar{\alpha}\Vert^2=4$, and at most two different root
  lengths are allowed in an irreducible root system.
\end{proof}

As a consequence of Proposition \ref{pp:pbc} we obtain
\begin{prop}\label{pp:moa}
Let $M(\mathfrak{g},\mathfrak{q}_{\Phi})$ be a
$\mathbf{G}$-orbit in the complex flag manifold $F$,
and $\mathfrak{Z}(M)$ the distribution of its $(0,1)$-vector fields.
Set
\begin{equation}
  \label{eq:mou}
  \mathcal{K}_{\Phi}=\left\{\alpha\in\mathcal{R}\mid
\mathfrak{g}^{\mathbb{C}}_{\alpha}\subset\mathfrak{k}_{\mathfrak{q}_{\Phi}}
\right\}.
\end{equation}
Then
\begin{equation}
  \label{eq:moua}
 \mathcal{K}_{\Phi}=\left\{\alpha\in\mathcal{Q} \left|
     \begin{gathered}
       \text{either } {-(\alpha+\bar\alpha)}
\notin\mathcal{Q}\\
\text{or } \mathbf{L}_{\alpha+\bar\alpha}\;\text{is indefinite}
     \end{gathered}\right\}\right. 
\end{equation}
and we have
\begin{equation}
  \label{eq:mov}
  \mathfrak{k}_{\mathfrak{q}_{\Phi}}
=\mathfrak{h}^{\mathbb{C}}\oplus
\sum_{\alpha\in 
\mathcal{K}_{\Phi}}\mathfrak{g}^{\mathbb{C}}_{\alpha}.
\end{equation}

A sufficient condition, in order that $\mathfrak{Z}(M)$
satisfy the higher Levi form concavity condition \eqref{eq:cra},
is that
\begin{equation}
  \label{eq:mot}\begin{cases}
  \text{for each }\gamma\in\mathcal{Q}^n_{\Phi}\;\text{with}\;
\gamma=\bar{\gamma}\;
  \text{and}\; \mathbf{L}_{\gamma}\geq{0},\\
\text{there exist}\;\alpha_0\in\bar{\mathcal{Q}}_{\Phi},\;\text{and}\;
\alpha_1,\hdots,\alpha_r\in\mathcal{K}_{\Phi}\cup\bar{\mathcal{K}}_{\Phi}
\;\,\text{s.t.}\\
\sum_{j=0}^h\alpha_j\in\mathcal{R}\;\text{for}\;\,1\leq{h}\leq{r},\;
\;-\gamma=\sum_{j=0}^r\alpha_j.
\end{cases}
\end{equation}
\end{prop}
\begin{proof}
  Denote by $\mathcal{H}^{\oplus}$ the set of real roots
  $\gamma\in\mathcal{Q}^n_{\Phi}\cap\bar{\mathcal{Q}}^n_{\Phi}$ such
  that the corresponding Levi form $\mathbf{L}_{\gamma}$ is
  semidefinite.\par 
  Consider any Levi form
  $\mathbf{L}=
  \sum_{\beta\in\mathcal{Q}^n_{\Phi}\cap\bar{\mathcal{Q}}^n_{\Phi}}c_\beta
  \mathbf{L}_\beta$, with $c_{\bar{\beta}}=\pm\bar{c}_\beta$. We claim
  that if $\mathbf{L}$ is semidefinite then 
  \[
  \bigcap_{\gamma\in\mathcal{H}^{\oplus}}\ker\mathbf{L}_\gamma\subset\ker\mathbf{L}.
  \]
  This claim implies \eqref{eq:moua} and \eqref{eq:mov}, from which
  the last statement follows.
\par
  With respect to the basis
  $\{Z_\alpha\}_{\alpha\in\mathcal{Q}\setminus\bar{\mathcal{Q}}}$, the
complexified 
  Levi forms $\mathbf{L}_\beta$ have mutually
  non overlapping nonzero entries, and each of them has at most one
  nonzero entry in each row (or column). Note that, for 
$\beta\in\mathcal{Q}^n_{\Phi}\cap\bar{\mathcal{Q}}^n_{\Phi}
\cap\mathcal{R}_{\mathrm{cx}}$, the matrix of $\mathbf{L}_{\beta}$ has no
diagonal entries.
  Partition
  $\mathcal{Q}^n_{\Phi}\cap\bar{\mathcal{Q}}^n_{\Phi}$ into the
  following subsets:
  \begin{equation*}
    \begin{aligned}
  \mathcal{H}_0&=
    \left\{\beta\in\mathcal{Q}^n_{\Phi}\cap\bar{\mathcal{Q}}^n_{\Phi}
\cap\mathcal{R}_{\mathrm{re}}\mid
      \mathbf{L}_\beta
      \text{ is semidefinite (hence diagonal)}\right\},\\      
\mathcal{H}_1&=
    \left\{\beta\in\mathcal{Q}^n_{\Phi}\cap\bar{\mathcal{Q}}^n_{\Phi}
\cap\mathcal{R}_{\mathrm{re}}\mid
      \mathbf{L}_\beta\text{ is indefinite
        and diagonal}\right\},\\
\mathcal{H}_2&=
    \left\{\beta\in\mathcal{Q}^n_{\Phi}\cap\bar{\mathcal{Q}}^n_{\Phi}\;\left|\;
        \mathbf{L}_\beta\;\;
  \text{has no
nonzero entry on the
          diagonal}
      \right.\right\},\\
\mathcal{H}_3&=
    \left\{\beta\in\mathcal{Q}^n_{\Phi}\cap\bar{\mathcal{Q}}^n_{\Phi}
\cap\mathcal{R}_{\mathrm{re}}
\left|
        \mathbf{L}_\beta\;\;
        \begin{gathered}
          \text{is indefinite and has nonzero}\\[-5pt]
\text{entries on and off the
          diagonal}
        \end{gathered}
\right.\right\}.
    \end{aligned}
  \end{equation*}
  If a coefficient $c_\beta$ in the decomposition of $\mathbf{L}$ is
  different from zero for some $\beta\in\mathcal{H}_1$, then
  $\mathbf{L}$ is indefinite. The same is true for
  $\beta\in\mathcal{H}_3$. Indeed, $\beta$ is real and there exists
  $\gamma$ such that $\beta=\gamma+\bar{\gamma}$ because
  $\mathbf{L}_{\beta}$ has a nonzero diagonal entry. Moreover, there
  are roots $\alpha,\alpha'\in\mathcal{Q}$ such that
  $\beta=\alpha+\bar{\alpha}'$ because $\mathbf{L}_{\beta}$ has a
  nonzero entry off of the main diagonal. By Lemma~\ref{lm:mmaa}, the
  restriction of $\mathbf{L}_{\beta}$ to
  $\mathfrak{g}^{\alpha}+\mathfrak{g}^{\alpha'}$ is indefinite.  Hence
  the only possible nonzero coefficients are those corresponding to
  roots in $\mathcal{H}_0=\mathcal{H}^{\oplus}$ and $\mathcal{H}_2$. 
  The statement then follows from the following statement about
  Hermitian matrices:\par
\textit{Let $A=(a_{i,j})_{1\leq{i,j}\leq\ell}$ 
be a Hermitian symmetric matrix with $a_{i,i}=0$ for $i=1,\dots,\ell$, and
$D=\mathrm{diag}(d_{1,1},\hdots,d_{\ell,\ell})$. If $D+A$ is
semidefinite then
$\ker{D}\subset\ker{D+A}$.}
\end{proof}
By using Proposition \ref{pp:da} we obtain
\begin{prop} Let $\mathfrak{g}$ be a semisimple real Lie algebra,
$\mathfrak{h}$ a 
Cartan subalgebra of $\mathfrak{g}$ and
$\mathcal{B}$ a system of simple roots of $\mathcal{R}^+$.
Let $\Psi\subset\Phi\subset\mathcal{B}$. \par
 If $\mathfrak{Z}(M(\mathfrak{g},\mathfrak{q}_{\Phi}))$
is subelliptic, then 
$\mathfrak{Z}(M(\mathfrak{g},\mathfrak{q}_{\Psi}))$ is
also subelliptic.\qed
\end{prop}
\begin{proof}
The inclusions $\mathfrak{q}_{\Phi}\subset\mathfrak{q}_{\Psi}$ 
and $\mathbf{Q}_{\Phi}\subset\mathbf{Q}_{\Psi}$ 
induce a smooth submersion $M(\mathfrak{q},\mathfrak{q}_{\Phi})\to
M(\mathfrak{q},\mathfrak{q}_{\Psi})$, that is also a $CR$ map and satisfies
the hypotheses of Proposition \ref{pp:da}.
\end{proof}

\subsection{Minimal orbits}
In \cite{Wolf69} it is shown that there is a unique $\mathbf{G}$ orbit
in $F$ (the \emph{minimal orbit})
that is compact. It is 
connected and has minimal dimension.
Minimal orbits
have been
studied, from the point of view of $CR$ geometry, in \cite{AMN06}.
If $(\mathfrak{g},\mathfrak{q})$ is the $CR$ algebra of a minimal orbit
$M$, then $\mathfrak{g}_{\mathbf{o}}$ contains a maximally vectorial
Cartan subalgebra $\mathfrak{h}$ of $\mathfrak{g}$.
This choice of $\mathfrak{h}$ is equivalent to the fact that 
$\mathcal{R}_{\mathrm{im}}=\mathcal{R}_{\bullet}$, i.e. that
all imaginary
roots are compact (see \cite{Ara62}).
\par
The fact that $M=M(\mathfrak{g},\mathfrak{q})$ is minimal is then
equivalent to the possibility of choosing the positive root system
$\mathcal{R}^+\subset\mathcal{Q}$ in such a way that
\begin{equation}
  \label{eq:mog}
\text{for}\;\alpha\in\mathcal{R}_{\mathrm{cx}}\quad
\alpha\succ{0}\Longleftrightarrow \bar{\alpha}\succ{0}.
\end{equation}
This leads to a complete classification of the minimal orbits in terms
of \textit{cross-marked Satake diagrams} (see e.g. \cite{AMN06}),
i.e. by systems $\Phi\subset\mathcal{B}$, where $\mathcal{B}$
are the simple roots of a positive system 
$\mathcal{R}^+$ satisfying \eqref{eq:mog}.
Using these $\Phi$'s, we give below a complete classification 
of the minimal orbits for which the distribution of $(0,1)$ vector
fields satisfies the higher order Levi concavity condition \eqref{eq:cra}.\par
When $\mathfrak{g}$ decomposes into a sum
$\mathfrak{g}=\mathfrak{g}^{(1)}\oplus\cdots\oplus\mathfrak{g}^{(\ell)}$
of simple ideals, the $CR$ manifold $M(\mathfrak{g},\mathfrak{q})$
is a Cartesian product 
$M(\mathfrak{g}^{(1)},\mathfrak{q}^{(1)})\times\cdots\times
M(\mathfrak{g}^{(\ell)},\mathfrak{q}^{(\ell)})$,
(see \cite[p.490]{AMN06}). Thus we can reduce the question of
the validity of the higher order Levi form concavity condition,
and also of the subellipticity and hypoellipticity of 
$\mathfrak{Z}(M(\mathfrak{g},\mathfrak{q}))$,
to the case where $\mathfrak{g}$ is a simple real Lie algebra.
Namely, this property will be valid for $M(\mathfrak{g},\mathfrak{q})$
if, and only if, it is valid for each factor 
$M(\mathfrak{g}^{(h)},\mathfrak{q}^{(h)})$, for $1\leq{h}\leq{\ell}$.\par
Thus we state the following classification theorem for the case of
a simple~$\mathfrak{g}$.\par
For the Satake diagrams characterizing the different simple real
Lie algebras, their labels and those of the roots of a simple positive
system,
 we refer to \cite{Hel78}, or to the Appendix in \cite{AMN06}.
\begin{thm}
  Let $(\mathfrak{g},\mathfrak{q}_{\Phi})$ be the $CR$ algebra of
a minimal orbit $M$, with $\mathfrak{g}$ simple and assume that
$M$ is a $CR$ manifold of finite type in the sense of \cite{BG77}.
 Then $\mathfrak{Z}(M)$
satisfies the higher order Levi form concavity condition 
\eqref{eq:cra} if and only if
one of the following conditions is satisfied
\begin{enumerate}
\item $\mathfrak{g}$ is either of the complex type, or compact, or of one
of the real types $\mathrm{AI\!{I}}$, $\mathrm{AI\!{I}\!{I}b}$,
$\mathrm{B}$, $\mathrm{CI\!{I}b}$, $\mathrm{DI}$, $\mathrm{DI\!{I}}$,
$\mathrm{DI\!{I}\!{I}a}$, $\mathrm{EI\!{I}}$, $\mathrm{EI\!{V}}$,
$\mathrm{EV\!{I}}$, $\mathrm{EV\!{I}\!{I}}$, $\mathrm{EI\!{X}}$;
\item $\mathfrak{g}\simeq\mathfrak{su}(p,q)$ 
is of type $\mathrm{AI\!{I}\!{I}a}-\mathrm{AI\!{V}}$
and $\Phi\cap\{\alpha_p,\alpha_q\}=\emptyset$;
\item $\mathfrak{g}\simeq\mathfrak{sp}(p,\ell-p)$, with $2p<\ell$,
is of type $\mathrm{CI\!{I}a}$ and either
$\Phi$ is all contained
in $\{\alpha_{2j-1}\mid 1\leq{j}<{p}\}
\cup\{\alpha_j\mid 2p<j\leq{\ell}\}$,
or $\Phi$ is all contained
in $\{\alpha_{2j-1}\mid 1\leq{j}\leq{p}\}$;
\item $\mathfrak{g}\simeq\mathfrak{so}^*(2\ell)$, with $\ell\in 2\mathbb{Z}_+$,
$
\Phi\cap\{\alpha_{\ell-1},\alpha_{\ell}\}=\emptyset$; 
\item $\mathfrak{g}$ is of type $\mathrm{EI\!{I}\!{I}}$ and
$\Phi\subset\{\alpha_3,\alpha_4,\alpha_5\}=\mathcal{R}_{\bullet}\cap
\mathcal{B}$;
\item $\mathfrak{g}$ is of type $\mathrm{FI\!{I}}$ and
$\Phi\subset\{\alpha_1,\alpha_2\}$.
\end{enumerate}
\end{thm}
Note that 
the condition for $M$ being 
of finite type in the sense of \cite{BG77} is explicitly described in terms
of $\Phi$ in \cite[Theorem 9.1]{AMN06}.
\begin{proof}
We use the results of \cite[\S 13, \S 14]{AMN06}. 
We stress the fact that we are assuming that $M$ is of finite type,
so that some choices of $\Phi$ are excluded from our consideration
because of \cite[Theorem 9.1]{AMN06}.
\par
First we prove that if \eqref{eq:cra} holds true, then 
$(\mathfrak{g},\mathfrak{q}_{\Phi})$ satisfies one of the conditions
$(1)\hdots(6)$. 
If
$(\mathfrak{g},\mathfrak{q}_{\Phi})$ is not one of those listed in the
statement, 
in all cases, with the two exceptions of $\mathrm{{C}I\!{I}a}$ with 
$\alpha_{2p-1}\in\Phi$ and $\Phi\cap\{\alpha_h\mid{h}>2p\}\neq\emptyset$,
and $\mathrm{{F}I\!{I}}$ with $\alpha_3\in\Phi$,
then $M$ admits a $\mathbf{G}$-equivariant fibration onto
one of the manifolds described in \cite[Examples~14.1, 14.2, and
14.3]{AMN06}. These are strictly pseudoconvex, hence $M$ does not
satisfy the higher order Levi form concavity condition by
Lemma~\ref{lm:ddb}. The two remaining cases 
will be discussed below, while proving the opposite implications.
\par\medskip
We know from \cite[Theorem 13.4]{AMN06} that $M$
is essentially pseudoconcave, and hence in particular
\eqref{eq:mot} and \eqref{eq:cra} hold,  in case
one of the following is satisfied:
\begin{enumerate}
  \item[$(1')$] case (1),
\item[$(2')$] case (2) with either $\Phi\subset\mathcal{R}_{\bullet}$,
or $\Phi\cap\mathcal{R}_{\bullet}=\emptyset$,
\item[$(3')$] case (3) with either
  $\Phi\subset\{\alpha_1,\dots,\alpha_{2p-1}\}$, or
  $\Phi\subset\{\alpha_{2p+1},\dots,\alpha_{\ell}\}$,
\item[$(4')$]  case (4),
\item[$(5')$]  case (5) when either 
$\alpha_4\in\Phi\subset\mathcal{R}_{\bullet}$, or $\Phi=\{\alpha_3,\alpha_5\}$;
\item[$(6')$] case (6) when $\Phi\subset\{\alpha_1,\alpha_2\}$.
\end{enumerate}
In all of these situations the subellipticity of
$\mathfrak{Z}(M)$ was already proved in \cite{HN00}.\par
To complete the proof, we need only consider the cases in the list 
in which
$M$ is not essentially pseudoconcave. Next we proceed to a case by case
discussion. 
\subsection*{$\mathrm{AI\!{I}\!{I}a}$} 
In this case $\mathfrak{g}\simeq\mathfrak{su}(p,q)$ with $2\leq{p}<q$.
Let $\Phi=\{\alpha_{j_1},\hdots,\alpha_{j_k}\}$. We need to discuss 
the case where $\Phi\cap\{\alpha_p,\alpha_q\}=\emptyset$,
but $\Phi$ intersects both $\mathcal{R}_{\bullet}$ and its complement
$\mathcal{B}\setminus\mathcal{R}_{\bullet}$. 
By the condition that $M$ is of finite type, we know from
\cite[Theorem 13.4]{AMN06} that
$\Phi$ does not contain at the same time
a simple root $\alpha_j$, with $1\leq{j}<p$, and its symmetrical
$\alpha_{p+q-j}$.
Since
$\Phi$ and $\Phi'=\{\alpha_{p+q-j}\mid\alpha_j\in\Phi\}$
define anti-isomorphic $CR$ manifolds, we can assume in the proof
that{\small
\begin{equation*}\begin{gathered}
  1\leq{j}_1<\cdots<j_a<p<j_{a+1}<\cdots<j_b<q<j_{b+1}<\cdots<j_k<j_{k+1}=p+q,\\
 p-j_a<j_{b+1}-q .
\end{gathered}
\end{equation*}}
The real roots that correspond to non zero semidefinite Levi forms are
\begin{equation*}
  \gamma_s=\sum_{j=s}^{p+q-s}\!\alpha_j\quad\text{for}\quad p+q-j_{b+1}<s\leq
j_a.
\end{equation*}
Indeed, those $\alpha\in\bar{\mathcal{Q}}_{\Phi}\setminus{\mathcal{Q}}_{\Phi}$
for which $\mathfrak{g}^{\mathbb{C}}_{\alpha}$ is not contained 
in the kernel of the Levi form $\mathbf{L}_{\gamma}$ 
must satisfy $\mathrm{supp}(\alpha)\cap\Phi=\mathrm{supp}(\gamma)\cap\Phi$.
\par
Thus all roots 
$\alpha$ in $\bar{\mathcal{Q}}_{\Phi}\setminus{\mathcal{Q}}_{\Phi}$
with $\mathrm{supp}(\alpha)\cap\Phi\not\supset\{\alpha_{j_a},\alpha_{j_b}\}$
belong to $\overline{\mathcal{K}}_{\Phi}$.
This is the case in particular for all the simple roots
not in $\mathcal{Q}$ and for the roots $-\sum_{j=p}^{q-1}\!\alpha_j$ and
$-\sum_{j=p+1}^{q}\!\alpha_j$, whose conjugates are $-\alpha_q$ and $-\alpha_p$,
respectively. This implies that all roots in 
$-(\mathcal{Q}^n_{\Phi}\cap\bar{\mathcal{Q}}^n_{\Phi})$ are sums of
roots in
$\mathcal{K}_{\Phi}
\cup\overline{\mathcal{K}}_{\Phi}$, giving condition \eqref{eq:mot}.
\subsection*{$\mathrm{CI\!{I}a}$}  We need to consider the cases where
$\Phi=\{\alpha_{j_1},\hdots,\alpha_{j_k}\}$ 
contains at the same time a root $\alpha_j$ with $j<2p$ and a root
$\alpha_j$ with $j>2p$. This means that
$k\geq{2}$, and we can order the indices in such a way that
\begin{equation*}
  1\leq\cdots<{j_a}<2p<{j_{a+1}}<\cdots\leq
\ell,
\end{equation*}
with $j_r$ odd for $1\leq{r}\leq{a}$. The positive real roots 
are
\begin{equation*}
  \gamma_r=\alpha_{2r-1}+\alpha_{\ell}+2\sum_{j=2r}^{\ell-1}\!\!\alpha_j.
\end{equation*}
Since their supports contain $\alpha_{j_{a+1}}$,
they all belong to $\mathcal{Q}_{\Phi}^n\cap\bar{\mathcal{Q}}_{\Phi}^n$.
Let $j_a=2q-1$.
The Levi forms $\mathbf{L}_{\gamma_r}$ are indefinite if $r>q$
by \cite[$(ii)$, p.520]{AMN06}.  Next we note that
$\mathbf{L}_{\gamma_r}$ with $1\leq{r}<q$ are identically zero, because
there is no root $\alpha$ in 
$\bar{\mathcal{Q}}_{\Phi}\setminus{\mathcal{Q}}_{\Phi}$ with
$\mathrm{supp}(\alpha)\cap\Phi=\mathrm{supp}(\gamma_r)\cap\Phi$.
Indeed, if $\alpha$ is a negative root whose support contains
$\mathrm{supp}(\gamma_r)\cap\Phi$ for some $1\leq{r}<q$, then
$\alpha_{2q-2}$ and $\alpha_{2q}$ both belong to $\mathrm{supp}(\alpha)$
and hence also to $\mathrm{supp}(\bar{\alpha})$. This implies that
also $\alpha_{2q-1}\in\Phi$ belongs to the support of 
both $\alpha$ and $\bar{\alpha}$,
showing that $-\alpha\in\mathcal{Q}_{\Phi}^n\cap\bar{\mathcal{Q}}_{\Phi}^n$.
\par
It was shown in \cite[p.520, $(iii)$]{AMN06} that $\mathbf{L}_{\gamma_q}$
is not zero and is positive semidefinite. The set of pairs 
$\beta,\beta'\in\bar{\mathcal{Q}}_{\Phi}\setminus{\mathcal{Q}}_{\Phi}$
with $\beta+\bar{\beta}'=\gamma_q$ was shown in \cite[p.520, $(iii)$]{AMN06}
to consist of the pairs $(\beta_s,\beta_s)$ with
\begin{equation*}
  \bar\beta_s=-\sum_{j=2q}^{s}\!\!\alpha_j\quad\text{for}\quad
2p\leq{s}<j_{a+1}.
\end{equation*}
We now distinguish two cases. If $q<p$, the root 
$-\bar\alpha_{2p}$ belongs to $\mathcal{K}_{\Phi}$. 
Then the root system generated by 
$\mathcal{K}_{\Phi}
\cup\overline{\mathcal{K}}_{\Phi}$
contains $-(\mathcal{Q}^n_{\Phi}\cap\bar{\mathcal{Q}}^n_{\Phi})$
and hence the higher order Levi form concavity condition 
\eqref{eq:mot} is satisfied.\par
When $p=q$, then no element of $(\mathcal{K}_{\Phi}\cup
\overline{\mathcal{K}}_{\Phi})\setminus\mathcal{R}^+$ contains $\alpha_{2p}$ 
in its support. Thus condition \eqref{eq:mot} fails for
$\gamma_q$ because, in the expression 
for $(-\gamma_q)$ in \eqref{eq:moh},
the coefficient of $\alpha_{2p}$ is $(-2)$, while for
all roots in $\bar{\mathcal{Q}}_{\Phi}$ it is $\geq(-1)$.

\subsection*{$\mathrm{EI\!I\!I}$} 
We only need to consider the cases where either $\Phi=\{\alpha_3\}$,
or $\Phi=\{\alpha_5\}$. Due to the symmetry of the $\mathrm{EI\!I\!I}$
diagram, we can restrict our attention to the case where $\Phi=\{\alpha_3\}$.\par
We note that $\gamma=\alpha_1+2\alpha_2+2\alpha_3+3\alpha_4+2\alpha_5+\alpha_6$
is the unique positive real root. It belongs to $\mathcal{Q}^n_{\Phi}\cap
\bar{\mathcal{Q}}^n_{\Phi}$ and the corresponding Levi form
$\mathbf{L}_{\gamma}$ is semidefinite, having rank $1$. 
Indeed, $\alpha=-\alpha_1-\alpha_2-2\alpha_3-2\alpha_4-\alpha_5$
is the unique root in $\bar{\mathcal{Q}}_{\Phi}\setminus{\mathcal{Q}}_{\Phi}$ 
for which 
$\alpha+\bar\alpha=-\gamma$, and there is no other pair
$\beta,\beta'\in\bar{\mathcal{Q}}_{\Phi}\setminus{\mathcal{Q}}_{\Phi}$ 
for which $\beta+\bar{\beta}'=-\gamma$.
Thus all simple roots 
belong to
$\mathcal{K}_{\Phi}
\cup\overline{\mathcal{K}}_{\Phi}$,
and hence the higher Levi form concavity condition \eqref{eq:mot} is satisfied.
\subsection*{$\mathrm{FI\!I}$}
We are left to discuss the case where $\{\alpha_3\}
\subset\Phi\subset\{\alpha_1,\alpha_2,\alpha_3\}$. There is only one
positive real root, namely
$\gamma=\alpha_1+2\alpha_2+3\alpha_3+2\alpha_4$, that belongs to
$\mathcal{Q}^n_{\Phi}\cap\bar{\mathcal{Q}}^n_{\Phi}$.  The
corresponding Levi form $\mathbf{L}_{\gamma}$ has rank one, and hence
it is semidefinite.  The set $\mathcal{K}_{\Phi}$ is the
complement of
$\{-{\alpha}_4\}$ in
$\mathcal{Q}_{\Phi}$. Then
no element of $(\mathcal{K}_{\Phi}\cup
\overline{\mathcal{K}}_{\Phi})\setminus\mathcal{R}^+$ contains $\alpha_{4}$ 
in its support. Thus condition \eqref{eq:mot} fails for
$\gamma$ because, in the expression 
for $(-\gamma)$ in \eqref{eq:moh},
the coefficient of $\alpha_{4}$ is $(-2)$, while for
all roots in $\bar{\mathcal{Q}}_{\Phi}$ it is $\geq(-1)$.
\end{proof}

\providecommand{\bysame}{\leavevmode\hbox to3em{\hrulefill}\thinspace}
\providecommand{\MR}{\relax\ifhmode\unskip\space\fi MR }
\providecommand{\MRhref}[2]{%
  \href{http://www.ams.org/mathscinet-getitem?mr=#1}{#2}
}
\providecommand{\href}[2]{#2}

\end{document}